\documentclass[a4paper,10pt]{amsart}
\usepackage[utf8]{inputenc}
\usepackage{tikz}
\usepackage{amsthm,amssymb}
\usepackage[
style=numeric,    
isbn=false,                
pagetracker=true,          
maxbibnames=50,            
maxcitenames=3,            
autocite=inline,           
block=space,               
backref=false,              
backrefstyle=three+,       
date=short,                
url=false,
doi=false,
]{biblatex}
\renewbibmacro{in:}{%
  \ifentrytype{article}{}{\printtext{\bibstring{in}\intitlepunct}}}

\bibliography{gKdV}  

\newtheorem{theorem}{Theorem}
\newtheorem{prop}[theorem]{Proposition}
\newtheorem{cor}[theorem]{Corollary} 
\newtheorem{lemma}[theorem]{Lemma}
\renewcommand{\d}{\partial}
\DeclareMathOperator{\sech}{sech}
\DeclareMathOperator{\Ai}{Ai}
\DeclareMathOperator{\Gi}{Gi}
\DeclareMathOperator{\Hi}{Hi}
\DeclareMathOperator{\real}{real}
\DeclareMathOperator{\im}{Im}
 
\newcommand{\R}{\mathbb{R}}

\title{Self similar solutions to super-critical gKdV}
\author{Herbert Koch}

\begin{document}

\begin{abstract}
We construct self similar finite energy solutions to the slightly 
super-critical generalized KdV equation. These self similar solutions 
bifurcate as a function of $p$ from the soliton at the $L^2$ critical exponent $p=4$.

\end{abstract}

\maketitle

\section{Introduction} 
Let $p\ge 1$ and consider the generalized KdV equation
\begin{align} \label{gkdv}
  \partial_{t}u+ \partial_{xxx}u \pm \partial_{x}(|u|^p u) &= 0  \\
  u(0,x) &= u_{0}(x) \nonumber \end{align}
 or, for integer exponents
$p$,
\begin{align} \label{gkdv2}  
\partial_{t}u+ \partial_{xxx}u \pm \partial_{x}(u^{p+1}) &= 0  \\
  u(0,x) &= u_{0}(x) \nonumber
\end{align} 
Both the KdV equation
($p=1$,\eqref{gkdv2}) and the mKdV equation ($p=2$,\eqref{gkdv2}) are integrable and in these
cases a remarkable amount of information can be obtained by the
inverse scattering machinery.  Both cases \eqref{gkdv}, (with $+$) and
\eqref{gkdv2} (either $p$ odd or $+$) admit soliton solutions $u(x,t) = Q_p(x-t)$ where
\begin{equation} 
  Q_p(x) =
 \left( \frac{p+2}2 \right)^{\frac1{p}}\sech^{\frac{2}p}\left(\frac{px}2\right). 
\end{equation}  
The quantities 
\[ 
\int_{}{ u dx},\quad \int_{}{ u^2 dx} \quad \mbox{and} \quad
\int_{}{\frac12 u_x^2- \frac1{p+2} |u|^{p+2}dx} 
\]
 are conserved. The
equations are invariant under translations in space and time, and
under the scaling
\[ u_\lambda (x,t) = \lambda^{2/p} u(\lambda x , \lambda^3 t), \]
which shows that the homogeneous Sobolev spaces $\dot H^s$ with index
 $s = 1/2 - 2/p$ are scale invariant; the quintic gKdV equation ($p=4$) 
has $L^2$ as critical space.

Kenig, Ponce and Vega \cite{MR1211741} prove local existence to \eqref{gkdv2} 
in the scaling critical  Sobolev space for all integers $p \ge 4$ and global existence for $ p=1,2,3$. This has been extended to critical Besov spaces 
by Molinet and Ribaud in \cite{MR2015413} and by Strunk \cite{MR3117359}  to \eqref{gkdv} for real  $p>4$.
This raises the question concerning global existence
and blow up in the critical and the supercritical case $p\ge 4$.
Numerical simulations by Dix and McKinney \cite{MR1664756} suggest
that there is self similar blow up in the supercritical case. In
contrast to the situation for NLS there is neither a virial identity
argument in the style of Glassey nor the explicit formula given by the
pseudo-conformal transformation. Nonetheless, Martel and Merle, and Martel, Merle and Rapha\"el showed
in a series of papers \cite{MR1896235,MR1888800,MR1944576,MR2055847,arxiv1204.4624,arxiv1204.4625,arxiv1209.2510} 
that in the $L^2$ critical case there are solutions which blow up
along the soliton manifold, i.e. the spatial scale of the solution
tends to zero in finite time.  Probably one of the earliest and most
prominent prediction of blow-up respectively wave collapse is due to Zacharov, Kuznetsov and Musher 
\cite{ZKM} for the (super critical) cubic focussing NLS in three
dimensions who write {\em Numerical simulations indicate that for d = 3 there
  is self-similar and spherically symmetric blow-up, even from
  non-symmetric initial data}. The blow-up mechanism for the nonlinear
Schr\"odinger equation is described in  detail in the book by
Sulem and Sulem \cite{MR1696311}. Indeed, Zakharov
\cite{za} predicted blow-up of the form
\[ 
 \frac{1}{L(t)}\exp{\big( \frac{1}{\tau(t)} \big)}Q\left(\frac{|x|}{L(t)};a
\right)
\]
for some selfsimilar profile $Q$ and the scaling parameters
\[  
L(t) = (2a(t^{*}-t))^{1/2} \qquad \mbox{and} \qquad \tau(t) = \frac{1}{2a}
\log{\left(\frac{t^*}{t^*-t}\right)},
\]
where $a>0$ is a specific parameter and $t^*$ is the time at which
blow-up occurs.  There seem to be solutions for each dimension $2 < d
\leq 3$ for one unique $a(d)$ and heuristic arguments in
\cite{MR966356} derive a relationship of the form
$$ d-2 \sim \frac{1}{a}\exp{\left(-\frac{\pi}{a}\right)}.$$

It seems that the first fully rigorous construction of self-similar
blow-up solutions is due to Kopell and Landman \cite{MR1349311} for
the cubic NLS in $\mathbb{R}^{2+\varepsilon}$ (which has to be
understood in the sense of existence of a solution to the nonlinear
ODE into which the dimension enters merely as a parameter).  It is
crucial that these solutions are in $\dot H^1 \cap L^{p+2}$ and hence
their energy vanishes.

Self similar solutions for gKdV have been constructed by Bona and
Weissler \cite{MR1705463}.  Their solutions are not in $\dot H^1$ and
their relation to the blow-up observed in simulations is not
clear. Such solutions can be obtained by evolving small self-similar
initial data, like for Navier-Stokes, or wave maps (Shatah et al).

In 2009, Merle, Raphael and Szeftel \cite{MR2729284} established blow
up from smooth initial data for NLS in the slightly super critical
case in low dimensions, heuristically {\em bifurcating} from the
soliton.

Here we construct selfsimilar solutions to the generalized KdV
equation for $p$ slightly larger than $4$ (Theorem \ref{selfsimilarsolution}). 
Moreover in Theorem \ref{im} we construct an almost invariant
manifold containing the solitons and the selfsimilar solutions, which
will play a central role in resolving the dynamic bifurcation at
$p=4$, together with fairly precise estimates in Theorem \ref{inverse} 
for the constructed
functions and their derivatives with respect to all parameters.

In Section \ref{sec:bifurcation} we formulate the bifurcation problem
and state the main technical result, Theorem \ref{inverse}, and the main
consequences, the existence result of Theorem \ref{selfsimilarsolution}
and  Theorem \ref{im}.
In Section \ref{sec:reduction} we deduce Theorem
\ref{selfsimilarsolution} and Theorem \ref{im} from Theorem \ref{inverse}.  
This reduction is elementary but  conceptually interesting.

The generalized Airy and Scorers functions are  studied in
Section \ref{sec:scorer} with standard arguments: Stationary phase,
contour integrals, and explicit formulas of Fourier transforms of
homogeneous functions.  The next Section \ref{sec:fundamental} derives
explicit formulas for a unique Green's function for the linear
part.

The next step (Section \ref{sec:weighted}) consists in a study of
estimates for integral operators with integral kernel related to the
Green's function. After this preparation we set up the inverse function
theorem in Section \ref{sec:weighted}.  Due to the weigths
differentiability with respect to $a$ is not
immediate. We approach it after establishing a fairly precise
asymptotic expansion (Section \ref{sec:expansion}) for the solution
constructed by the inverse function theorem.

The final section shows plots of numerically computed self similar solutions
for various values of $a$ and $p$ due to Strunk \cite{strunk}. I want to thank 
Nils Strunk for allowing me to include this data and S. Steinerberger for 
many discussions.

\section{The bifurcation problem} 
\label{sec:bifurcation} 
We search for self-similar solutions $\psi(t,x)$ of the form
\begin{equation} \label{ansatz} \psi(t,x) =
  (3t)^{-\frac2{3p}}v\left(\frac{x}{(3t)^{1/3}} \right),
\end{equation}
for which the self similar profile $v$ has to satisfy
\begin{equation} \label{selfsimilar}
\frac2p v + y v_y - v_{yyy}  - (|v|^{p}v)_y = 0.
\end{equation} 
A change of coordinates leads to a formulation in which the bifurcation 
from the soliton equation becomes visible:
let $a>0$, 
\[ y = a^{1/3}(x+ a^{-1}), \qquad  u(x) =   a^{\frac2{3p}} v( a^{1/3} (x + a^{-1})) \]
then \eqref{selfsimilar} is equivalent to 
\begin{equation} \label{bifurcationa} 
a \left( \frac2p u + x u_x\right) - u_{xxx} + u_x - (|u|^{p}u)_x = 0. 
\end{equation}  
Reversing the derivation, if $u$ satisfies \eqref{bifurcationa} then
\begin{equation} v(y) = a^{-\frac2{3p}} u(a^{-1/3} y -
  a^{-1}) \label{uv} \end{equation} 
is a solution for
\eqref{selfsimilar} and we thus get via \eqref{ansatz} a self-similar
solution for \eqref{gkdv}.  We will construct self similar solutions
in $L^{p+2}$ with derivative in $L^2$. Since for any solution of gKdV
the quantity
\[ \int_{\mathbb{R}}{\frac{1}{2}u_x^2 - \frac{1}{p+2}|u|^{p+2}dx}\]
is formally conserved, plugging the ansatz into gKdV one sees  that the existence of the
integral already implies it being 0 for all times.

For $a=0$ the equation simplifies to the derivative of the soliton equation
\begin{equation} \Big(- u_{xx} + u - (|u|^{p}u)\Big)_x =
  0, \end{equation} 
which motivates searching for a branch of
solutions bifurcating from the soliton $Q_{p}$ using $a$ as
bifurcation parameter (see also Sulem and Sulem \cite{MR1696311}).

This is not yet the complete picture and complications arise from the
linearization around the soliton
\begin{equation} 
\label{opl}  L \psi := -\psi_{xx} + \psi - (p+1) Q_p^p \psi 
\end{equation} 
being elliptic but not invertible. Its spectrum, however, is
explicitly known: there is a ground state $Q_p^{\frac{p}2+1}$ and the
second eigenvalue is 0 with an eigenspace spanned by $Q_p'$.  We
search $p$ and $u$ as functions of $a$. This requires an additional
normalization which we choose to be
\begin{equation} \label{orth} \langle u, Q'_p \rangle = 0. \end{equation}
Our considerations lead to the bifurcation formulation
\begin{equation} \label{bifurcation}
 a \left( \frac2p u + x u_x
  \right) - u_{xxx} + u_x - (|u|^{p}u)_x+ \langle Q'_p,u \rangle Q_p'' =  0.
\end{equation}   
It will be useful to consider a generalization which will give an approximate 
invariant manifold which contains both, the solitons, and the selfsimilar blow up solutions.
We consider  

\begin{equation} \label{bifurcationsigma}
 a \left( (1+\gamma)  u + x u_x
  \right) - u_{xxx} + u_x - (|u|^{p}u)_x+ \langle Q'_p,u \rangle Q''_p =  0.
\end{equation}

\begin{theorem}\label{inverse}  
Let $q >4$. There exists $\varepsilon > 0$ and a unique map 
\[ u \in C^\infty\left([0,\varepsilon)\times (-\frac12-\varepsilon, -\frac12+\varepsilon)\times (3,q)\times \R  \right)  \]
with the following properties: 

\begin{equation} 
u_{a,\gamma,p}(x) \text{ satisfies \eqref{bifurcationsigma} for }
  0\le a < \varepsilon, |\frac12 +\gamma|  < \varepsilon, 2\le  p\le  q
\end{equation}
\begin{equation}\label{point}
  \sup_{a,\gamma,p,x} (1+a|x|)^{1+\gamma} |u_{a,\gamma,p}(x) | < \infty, 
 \end{equation} 
\begin{equation}
  \sup_{a,\gamma,p,x}   (1+ x_+)^{1-k} | \partial_\gamma^n   \partial_p^m \partial_a^k u_{a,\gamma, p}(x)| < \infty 
\end{equation} 
for $k \ge 1$,
\begin{equation}
 u_{0,\gamma,p}(x) = Q_p(x) 
\end{equation} 
\begin{equation}\label{ener}  u_{a,\gamma,p}(x) >0, \qquad \partial_x u_{a,\gamma,p}\in L^2(\R)
\end{equation} 
The solution $u_{a,\gamma,p}$ is the unique solution 
to \eqref{bifurcationsigma} satisfying \eqref{point} and \eqref{ener} 
in a small neighborhood of the soliton. 
\end{theorem} 

The main results are consequences. 

\begin{theorem} \label{selfsimilarsolution}
  There exists $\varepsilon >0$ and a unique function $p  \in
  C^\infty([0,\varepsilon))$ with
 \[ 
 p(0)=4, \qquad \frac{d p}{da}(0)  = 
\frac{\Vert Q \Vert_{L^1}^2} {\Vert Q \Vert_{L^2}^2} = \frac{\Gamma(1/4)^4}{4 \pi^2} \sim 4.3768 \dots  
\]
such that $x \to u_{a,\frac2{p(a)}-1, p(a)}(x)$ is a solution to \eqref{bifurcationa} with 

\begin{equation} 
\partial_x u_{a,\frac2{p(a)}-1, p(a)} \in L^2,   \quad  \sup (1+a|x|)^{1+\gamma} |u_{a,p(a) } | <\infty 
\end{equation} 
and
\begin{equation} 
 E(u_{a,\frac{2}{p(a)}-1,p(a)}) := \int \frac12
  (\partial_{x}u_{a,\frac{2}{p(a)}-1, p(a) })^2- \frac1{p+2} |u_{a,\frac2{p(a)}-1, p(a) }|^{p+2} dx = 0
\end{equation}   
\end{theorem}

These solutions are contained in a family of solutions which contains 
the solitons and the selfsimilar solution.

\begin{theorem}\label{im}
Let $q>4
$.
  There exists $\varepsilon >0$ and a unique function $\gamma(a,p)  \in
  C^\infty([0,\varepsilon)\times [3,q])$ with
 \[ 
 \gamma(0,4)=-\frac12, \gamma(a,p(a)) = \frac2p-1, \]
 \[ \frac{\partial \gamma}{\partial a}(0,4) = \frac18\frac{\Vert Q \Vert_{L^1}^2} {\Vert
   Q \Vert_{L^2}^2} = \frac18\frac{\Gamma(1/4)^4}{4 \pi^2} \sim
 \frac18 4.3768 \dots,
\]
\[ \frac{\partial \gamma}{\partial p }(0,4) =   0   \]
such that $x \to u_{a,\gamma(a,p), p}(x)$ is a solution to 
\begin{equation} 
a( (1+\gamma(a,p))  u + x u_x) - (u_{xx} - u + |u|^p u)_x = 0  
\end{equation}
 with 
\begin{equation} 
\partial_x u_{a,\gamma(a,p), p} \in L^2,  \quad  \sup(1+a|x|)^{1+\gamma} |u_{a,\gamma(a,p), p }| \le c . 
\end{equation} 
Moreover $u_{0,\gamma(0,p),p}=Q_p$.
\end{theorem}  

In the process of proving Theorem \ref{inverse} we obtain fairly precise asymptotics 
for  the constructed solutions. This asymptotics can be expressed concisely 
in terms of the special functions $\Hi_\gamma$ and $\Gi_\gamma$ constructed in 
Section \ref{sec:scorer}.

\section{Theorem \ref{inverse} implies Theorem \ref{selfsimilarsolution} and Theorem \ref{im}  }
\label{sec:reduction}

\subsection{The soliton} 

We recall that solitons $Q$ satisfy, possibly after rescaling,  
\begin{equation} \label{soliton} -Q_{xx} + Q - |Q|^p Q =
  0. \end{equation} 
There is a unique solution, up to the choice of sign and a translation parameter. 
We denote by $Q$ (or $Q_p$)  the unique symmetric and nonnegative solution. 
 We multiply by $Q$ and $xQ_x$, respectively, and integrate to
obtain the identities
\begin{equation}
  \int Q_x^2+ Q^2- Q^{p+2} dx = 0 =
 \int \frac12 Q_x^2- \frac12 Q^2+ \frac1{p+2}Q^{p+2} dx. 
\end{equation} 
This implies 
\begin{equation} 
\Vert Q \Vert_{L^{p+2}}^{p+2} = \frac{2(p+2)}{p+4} \Vert Q \Vert_{L^2}^2 , 
\quad \Vert Q_x \Vert_{L^2}^2 =  \frac{p}{p+4} \Vert Q \Vert_{L^2}^2 
\end{equation} 
and hence 
\begin{equation} 
  E(Q) = \int  \frac12 Q_x^2 - \frac1{p+2} Q^{p+2} dx = 
 \frac{p-4}{2(p+4)} \Vert Q \Vert_{L^2}^2  
\end{equation} 
from which we see that the energy vanishes if $p=4$. Let $Q_c (x) =
c^{-2/p} Q(x/c)$, which is a rescaling of the soliton so that
\begin{equation} \label{eq:travelling}  u(x,t) = Q_c(x-c^2 t)   \end{equation} 
is a traveling wave solution to the gKdV equation with speed $c^2$. Then 
\[ 
\tilde Q_c := 
- c\frac{\partial}{\partial c } Q_c = \frac2p Q_c + \frac{x}{c} Q_c'
 \]
satisfies  
\[ \Vert Q_c \Vert_{L^2}^2 = c^{1-\frac4p} \Vert Q \Vert^2_{L^2}, \]
hence, using the notation
$\tilde Q = \tilde Q_{1}$,
\begin{equation} 
  \langle \tilde Q, Q \rangle = -\frac12 \left.  \frac{d}{dc}  
\Vert Q_c \Vert_{L^2}^2  \right|_{c=1} = 
\left( \frac2p-\frac12\right) \Vert Q \Vert^2_{L^2} 
\end{equation} 
which changes sign as $p$ passes through $4$.  We differentiate \eqref{eq:travelling} with respect to $c$,  evaluate at $c=1$ and
obtain a solution to the linearized  equation,  
hence 
\[ \frac{d}{dx}  (- 2Q - L \tilde Q ) = 0 \]
and  
\begin{equation} L \tilde Q = -2 Q. \end{equation} 

 An integration by parts gives 
\begin{equation}  
\int \tilde Q dx = \int \frac2p Q + x Q_x dx = \left( \frac2p-1\right)
\int Q dx.
\end{equation} 

\subsection{The derivatives with respect to $a$} 
Let $\dot v$ be the derivative of $u$ with respect to $a$ evaluated at
$a=0$. It decays at $-\infty$ and hence it satisfies
\begin{equation}\label{dotv} 
    L  \dot v + \langle \dot v, Q_x \rangle Q_x = 
- \int_{-\infty}^x (1+\gamma)  Q + x Q_x dy 
\end{equation} 
We multiply  by $Q_x$ (supressing $p$ in the notation) and, since  $LQ_{x} = 0$, and 
\begin{equation}\label{da} 
 \langle \dot v, Q_x\rangle = \Vert Q_x
  \Vert_{L^2}^{-2} \langle Q,\tilde Q \rangle = \left( \gamma+\frac12
  \right) \frac{\Vert Q \Vert_{L^2}^2}{\Vert Q_x \Vert_{L^2}^2}.
 \end{equation} 
 Observe that  $\langle \dot v, Q_x \rangle = 0$ if $\gamma=-\frac12$. 
The norms on the right hand side can be evaluated and this  gives the derivative of  the inner product with respect to $a$ at $a=0$ 
as a function of $p$ and $\gamma$. 

We set $\gamma=-\frac12$,  multiply \eqref{dotv}  by $\dot v$ and integrate
\[ 
\langle \frac12 Q + x Q_x, \dot v \rangle + \int \dot v_x \dot v dx
+ (p+1) \langle Q^p \dot v ,\dot v_x \rangle = 0. 
\] 
We rewrite the  middle integral as a limit
\[
 \int \dot v_x \dot v dx =    \lim_{R \to \infty}
  \int_{-\infty}^R  \dot v_x \dot v \, dx
  =   \lim_{R\to \infty} \frac12 (\dot v (R))^2
 \] 
This limit can be calculated: the inverse of  $-\partial_{xx} + 1$ is given 
by the convolution by $\frac12 e^{-|x|}$. It maps the constant function $1$ to itself, hence 
\[  
\lim_{R\to \infty} \frac12 (\dot v (R))^2  =   \frac12 (\int \tilde Q dx)^2   
 =  \frac18 \left( \int Q dx \right)^2 
\]
and  
\begin{equation}\label{idinfty}  
  \langle \frac12 Q + x Q_x, \dot v \rangle 
  + (p+1) \langle Q^p \dot v ,\dot v_x \rangle = 
-\frac18 \left( \int Q dx \right)^2. 
\end{equation}

Let $\ddot v$ be the
second derivative with respect to $a$ evaluated at $a=0$. It satisfies
\[
 2 (\frac12 \dot v + x \partial_x \dot v ) + \partial_x \left(L \ddot v -
p(p+1) Q^{p-1} \dot v^2 + \langle \ddot v, Q_x \rangle Q_x\right ) = 0
\] 
 We  fix $p=4$, multiply by $Q$ and integrate. Then, since 
\[ \langle \dot v , \frac12 Q+xQ_x\rangle+ \langle \frac12 \dot v+ x \dot v_x, Q \rangle=0, \]
 and using  \eqref{idinfty}, 
\[ 
\begin{split} 
  \Vert Q_x \Vert_{L^2}^2 \langle \ddot v, Q_x\rangle = & 2 \langle
  \frac12 \dot v+ x \partial_x \dot v, Q \rangle + 20 \int Q^3 Q_x
  \dot v^2 dx 
\\ = & - 2 \langle \dot v, \frac12 Q+ xQ_x \rangle - 10 \int
  Q^4 \dot v \dot v_x dx 
\\ = & \frac14 \Vert Q \Vert_{L^1}^2
\end{split}  
\]
and hence the second derivative of the inner product with respect to $a$ 
at $a=0$, $\gamma=-\frac12$ and $p=4$ is  given by  
\begin{equation} \label{daa} 
\langle \ddot v, Q_x\rangle = \frac14
  \frac{\Vert Q \Vert_{L^1}^2}{ \Vert Q_x \Vert_{L^2}^2}.
\end{equation}

We define the smooth function 
\[ 
(a,\gamma, p) \to \eta(a,\gamma, p) := \langle u_{a,\gamma, p},\partial_x Q_{p} \rangle 
\]
on 
$[0,\varepsilon)\times (-\frac12 - \varepsilon,
\frac12+\varepsilon) \times [2,q]$.  The orthogonality $ \langle Q ,
Q_x \rangle = 0 $ implies
\[  \eta(0,\gamma, p) =  0, \]
the derivative with respect to $a$ 
is given by \eqref{da}
\[
\frac{\partial \eta }{\partial a}(0,\gamma,p) = \left( \gamma+\frac12
\right) \frac{\Vert Q \Vert_{L^2}^2}{\Vert Q_x \Vert_{L^2}^2},
\]
hence 
\begin{equation} \frac{\partial \eta}{\partial a} (0,-\frac12, p )  = 0 \end{equation}
and 
\[ 
\frac{\partial^2\eta }{\partial a \partial \gamma  }(0,\gamma,p) =  
 \frac{\Vert Q \Vert_{L^2}^2}{\Vert Q_x \Vert_{L^2}^2}.  
\]
We read the second derivative with respect to $a$ from \eqref{daa} 
\begin{equation} 
  \frac{\partial^2\eta}{\partial a^2}(0,-\frac12 , 4) = 
 \frac14 \frac{\Vert Q \Vert_{L^1}^2} {\Vert Q_x \Vert_{L^2}^2}.
\end{equation} 

Let
\begin{equation}  g(a,\gamma, p) = \eta(a,\gamma,p)/a
\end{equation} 
which  is a smooth function with (again we suppress $p$ in the notation of $Q$) 
\begin{eqnarray*}   g(0,\gamma,p) & = &  0, 
\\ \frac{\partial g}{\partial a}(0,-\frac12,4) &  = &  
  \frac18 \frac{\Vert Q \Vert_{L^1}^2} {\Vert Q_x \Vert_{L^2}^2}
\\
\frac{\partial g} {\partial \gamma } (0,-\frac12, p) & = &   \frac{\Vert Q \Vert_{L^2}^2}{\Vert Q_x \Vert_{L^2}^2} 
\\ 
\frac{\partial g} {\partial p } (0,-\frac12, p)
& =& 0  
  \end{eqnarray*} 
and by the implicit function theorem the equation 
\[ g(a,\frac2p-1,p) = 0 \]
can be solved for $p= p(a)$ for $a \in [0,\varepsilon)$, possibly after decreasing $\varepsilon$ if necessary.
Clearly $p(0)= 4$ and 
\[
  \frac{dp}{da}(0) = 
  -\left.  \frac{ \frac{\partial g}{\partial a}}{\frac{\partial g }
 {\partial p}} \right|_{a=0,p=4}  
 = 
  - \left. \frac{ \frac{1}{2}\frac{\partial^2 \eta}{\partial a^2}}
{\frac{\partial^2 \eta}{\partial a \partial p}}  \right|_{a=0,p=4}
 =   \frac{\Vert Q \Vert_{L^1}^2} {\Vert Q \Vert_{L^2}^2}
\]
We recall
\[ Q(x) = 3^{1/4} \sech^{1/2}(2x) \]  
and thus 
\[ 
\int Q dx =
\frac{3072^{\frac{1}{4}}}{\sqrt{\pi}}\Gamma\left(\frac{5}{4}\right)^2, 
\]
\[
 \int Q^2 dx =   \frac{\sqrt{3}}{2}\pi 
\]
and hence 
\[
 \frac{dp}{da}(0) =     \frac{64}{\pi^2}\Gamma\left(\frac{5}{4}\right)^4=  \frac1{4\pi^2} \Gamma(1/4)^4 \sim 4.3768\dots.
 \]
The changes for Theorem \ref{im} are quite obvious: We solve 
\[ g(a,\gamma, p ) = 0   \]
for $\gamma(a,p)$  near $(0,-\frac12, 4)$ and obtain 
\[ \frac{\partial}{\partial a} \gamma(0,4) =  \frac1{32\pi^2} \Gamma(1/4)^4 \sim 
 0.54711 \]
and 
\[ \frac{\partial}{\partial p} \gamma(0,4) = 0. \]

\subsection{ Vanishing energy}  
The function $u_{a,\frac2{p(a)}-1,p(a)}$ satisfies
 \eqref{bifurcation} and \eqref{orth}, hence
 \eqref{bifurcationa}. Moreover
\[ |u_{a,\frac{2}{p(a)}-1,p(a)} | \le c (1+ |x|)^{-2/p} \in L^{p+2} \]
for $p \ge 1$.
By Theorem \ref{inverse} the derivative with respect to $x$ is in $L^2$.
We observed above that then the energy has to vanish.

This completes the proof that Theorem \ref{inverse} implies Theorem
\ref{selfsimilarsolution} and Theorem \ref{im}.

\section{The Airy function and Scorer's functions}
\label{sec:scorer} 
\subsection{Definition and first properties}

In this section we study a class of special functions closely related
to the Airy function. The Airy function and Scorers functions are discussed in 
\cite{MR2723248}, and the   notation is motivated by Dix
\cite{MR1479638} but with deliberate essential changes. 
We define for $\gamma \in \mathbb{C}  $ with real part larger than $-1$ 
\[ 
\begin{split} 
  \Ai_\gamma (x) = & \frac1{2\pi} \real  \int_{-\infty}^{\infty}  
(\sigma/i)^\gamma e^{i( \sigma^3/3 + x \sigma)} d\sigma \\
  = & \frac1{\pi} \int_0^\infty \sigma^\gamma 
 \cos( \frac13 \sigma^3+ x \sigma - \frac{\gamma \pi}2) d\sigma \\
  = & \frac1\pi \int_0^\infty \sigma^\gamma \left( \cos(\frac{\gamma
      \pi}2) \cos(\frac13 \sigma^3 + x \sigma) + \sin( \frac{\gamma
      \pi}2) \sin(\frac13\sigma^3 + x \sigma)\right) d\sigma.
\end{split} 
\]

Clearly  $\Ai_\gamma$ depends holomorphically on $\gamma$.

The first line of the equation defines $\Ai_{\gamma}$ through the Fourier transform. The
second line is the corresponding real formulation ( if $\gamma$ is real) and the last line
connects the definition to the slightly different ones in
\cite{MR1479638}.  We easily see that

\begin{equation}\label{Aider}  
   \Ai_\gamma' = -A_{\gamma+1} 
\end{equation} 
and  
\begin{equation} 
(1+\gamma) \Ai_\gamma + x \Ai_{\gamma}' -
  \Ai_\gamma''' = 0. 
\end{equation} 
This identity can be rewritten as
\begin{equation} (1+\gamma) \Ai_\gamma - x \Ai_{\gamma+1} +
  \Ai_{\gamma+3} = 0, \end{equation}
moreover,
\[ \Ai_0 = \Ai. \]

It is not hard to evaluate the function $\Ai_\gamma$ at $x=0$
\begin{equation}\label{Ai0} 
\begin{split} 
  \Ai_{\gamma}(0) = & \frac1{\pi} \real \int_{0}^{\infty}
  (\sigma/i)^\gamma e^{i\sigma^3/3 } d\sigma \\ = & \frac1{\pi}
  3^{\frac{\gamma-2}3}  e^{-\frac{(\gamma-2)\pi}3} \im\int_0^\infty
  \mu^{\frac{\gamma-2}3} e^{-\mu} d\mu \\ = & -\frac1{\pi} \sin(
  \frac13 \pi (\gamma-2)) 3^{\frac{\gamma-2}3}\Gamma((\gamma+1)/3 ).
\end{split} 
\end{equation} 

We work out the asymptotic behavior using the standard approach via
contour integration and stationary phase.  If $x < 0$ we apply stationary
phase and shift the contour around zero to the upper half plane
so that the leading contribution comes from the stationary point $\xi =
(-x)^{1/3}$. We obtain the leading term
\begin{equation} \Ai_\gamma 
\sim  \frac1{\sqrt{\pi}}  |x|^{-\frac14+\frac{\gamma}2}    \cos( \frac23 |x|^{3/2} - \frac{\pi}4 - \frac{\gamma \pi }2 ) 
 \end{equation} 
as $ x \to -\infty$. More precisely 
\begin{equation} 
  \Ai_\gamma(x)  
  =     \real \left\{ \left(  \frac1{\sqrt{\pi}}  |x|^{-\frac14+\frac{\gamma}2} 
 +   O( |x|^{-\frac74+ \frac{\gamma}2} )\right) e^{i (\frac23 |x|^{3/2} - \frac{\pi}4 -
 \frac{\gamma \pi }2)} \right\}  
 \end{equation} 
 as $x\to - \infty$ and  we can replace
 $O(|x|^{-\frac74+\frac{\gamma}2})$ by an asymptotic series
\begin{equation} \label{squareroot}
 |x|^{-\frac74 +\frac{\gamma}2} \sum_{j=0}^\infty  c_j |x|^{-3j/2}. 
\end{equation}  
We turn to $x >0$, shift the contour of integration to $ \mathbb{R} +
i\sqrt{x}$ and obtain again by stationary phase
\begin{equation}\label{airyasym} 
\Ai_\gamma  = 
 \left( \frac1{2 \sqrt\pi}  |x|^{-\frac14+\frac{\gamma}2}
+ O( |x|^{-\frac74+ \frac{\gamma}2} ) \right)     e^{-\frac23 x^{3/2} } 
 \end{equation} 
 as $ x \to \infty$. Again the $O(|x|^{-\frac74 +\frac{\gamma}2 })$ can be replaced by an asymptotic
 series \eqref{squareroot}.  These series can  be differentiated term by term with 
respect to $\gamma$, with the expected estimates for the difference 
of $\Ai_\gamma$ to the partial sum. 

Similarly, we set for the same set of $\gamma$
\[ 
\begin{split} 
  \Gi_\gamma (x) = & \frac{1}{\pi} \im \int_0^{\infty}
  (\sigma/i)^\gamma e^{i ( \frac13 \sigma^3+ x\sigma)} d\sigma
  \\ = & \frac1\pi \int_0^\infty \sigma^\gamma \sin( \frac13 \sigma^3 
+ x \sigma - \frac{\gamma \pi}2) d\sigma \\
  & \frac1\pi \int_0^\infty \sigma^\gamma \left( -\sin(\frac{\gamma
      \pi}2) \cos(\frac13 \sigma^3 + x \sigma) + \cos( \frac{\gamma
      \pi}2) \sin(\frac13\sigma^3 + x \sigma)\right) d\sigma.
\end{split} 
 \]
Again it is easily seen that
\[ \Gi_\gamma' = -\Gi_{\gamma+1} \]
and 
\[ (1+\gamma) \Gi_\gamma + x \Gi'_{\gamma} - \Gi'''_\gamma = 0 \]
which we can again rewrite as 
\[ (1+\gamma) \Gi_\gamma - x \Gi_{\gamma+1} + \Gi_{\gamma+3} = 0. \]

Evaluation at zero gives
\begin{equation} \label{Gi0} 
\Gi_\gamma (0)  =   \frac{1}{\pi}  \im 
\int_0^{\infty} (\sigma/i)^\gamma e^{i  \frac13 \sigma^3} d\sigma
 =   -\frac1{\pi} \cos( \frac{\pi}3  (\gamma-2)) 3^{\frac{\gamma-2}3}
\Gamma((\gamma+1)/3 ) 
\end{equation}

There are two contributions for large $x$, one from the integral near
zero and a second one from the oscillatory part. We choose a smooth cutoff
function supported in $|\sigma| \le 2$, identically $1$ in $|\sigma|
\le 1$ and we write
\[ \begin{split} 
\Gi_\gamma (x) =& \Gi_\gamma^s + \Gi_\gamma^0 
\\ =  &  \frac{1}{\pi} \im \int_0^{\infty}
 \eta(\sigma) (\sigma/i)^\gamma e^{i ( \frac13 \sigma^3+ x\sigma)}
 d\sigma 
\\ & + \frac{1}{\pi} \im \int_0^{\infty} (1-\eta(\sigma))
 (\sigma/i)^\gamma e^{i ( \frac13 \sigma^3+ x\sigma)} d\sigma.
\end{split} 
\]
Then 
\begin{equation} 
\begin{split} 
\Gi_\gamma^s(x) = & \sum_{j=0}^\infty \frac1{j!} (-1/3)^j \frac1\pi \im \int_0^\infty
  (\sigma/i)^{\gamma+3j}  e^{i x\sigma}\eta(\sigma)  d\sigma
\\ = & 
\sum_{j=0}^{\infty} \frac{(-1/3)^j}{\pi j!} \int_0^\infty (\sigma/i)^{\gamma+3j} e^{ix\sigma} d\sigma 
+ O(|x|^{-\infty}). 
\end{split} 
\end{equation} 
in the sense of oscillatory integrals. 
Now suppose that $x >0$. Then we move the contour of integration to $ i \R_+$:
\begin{equation} 
\int_0^\infty (\sigma/i)^\mu e^{ix\sigma} d\sigma = \int_{i \R_+} (\sigma/i)^\mu e^{ix\sigma} d\sigma 
= i \int_0^\infty t^\mu e^{-xt } dt 
= i x^{-1-\mu} \Gamma(1+\mu) 
\end{equation}  
If $ x < 0$ we move the contour to $-i \R_+$ and obtain for $\Gi_\gamma^s(x)$
\begin{equation} 
 \sum_{j=0}^\infty   \frac{(-1/3)^j\Gamma(1+\gamma+3j)}{j!\pi} |x|^{-1-\gamma-3j} \left\{
    \begin{array}{ll} -\cos(\pi(\gamma+3j) & x<0 \\ 1 & x>0 \end{array}
  \right.  + O(|x|^{-\infty} ).
 \end{equation}

The oscillatory part (for $x< 0$ ) is dealt with as above and 
we obtain the leading term 
\begin{equation}
\Gi_\gamma^o \sim  
 -   \frac1{\sqrt{\pi}}  |x|^{-\frac14+\gamma/2}    \sin( \frac23 |x|^{3/2} - \frac{\pi}4 - \frac{\gamma \pi }2 )  
 \end{equation} 
as $ x \to -\infty$, again with the same type of asymptotic series,  and 
it is $O(|x|^{-\infty})$  as $ x\to \infty$. Again it can be differentiated term by term with respect to $x$ and $\gamma$.

Finally, we set for $\gamma> -1$ 
\[
 \Hi_\gamma(x) = \frac1{\pi} \int_0^\infty \sigma^\gamma e^{-\frac13
  \sigma^3+ \sigma x } d \sigma.  
\] 
The derivative is again simple
\[ \Hi_\gamma' = \Hi_{\gamma+1} \]
and furthermore
\[ (1+\gamma) \Hi_\gamma + x\Hi_\gamma' - \Hi'''_\gamma = 0 \]
which we rewrite as 
\[ (1+\gamma) \Hi_\gamma + x\Hi_{\gamma+1} - \Hi_{\gamma+3} = 0 \]

The evaluation at $x=0$ is given by
\begin{equation}\label{Hi0}  
 \Hi_\gamma(0) =   \frac1{\pi} \int_0^\infty  \sigma^\gamma e^{-\sigma^3/3} d\sigma
=  \frac1{\pi} 3^{\frac{\gamma-2}3}   \int  \rho^{(\gamma-2)/3}    e^{-\rho} d\rho
 =  \frac1{\pi} 3^{\frac{\gamma-2}3} \Gamma( (\gamma+1)/3).  
\end{equation}

It is not hard to see that
\[
 \Hi_\gamma(x) = \sum_{j=0}^\infty\frac{\Gamma(1+\gamma+3j) }{3^j j!\pi} |x|^{-1-\gamma-3j}
+ O(|x|^{-\infty})  \]
as $ x \to -\infty$ and 
\begin{equation} \label{hiasym}  
\Hi_\gamma(x) = \left\{ \frac1{\sqrt{\pi}} x^{-\frac14 +
    \frac{\gamma}2}+ O(x^{-\frac74+ \frac{\gamma}2}) \right\}
e^{\frac23 x^{3/2}} 
\end{equation}  
as $x \to \infty$, where again the $O(|x|^{-\frac74+\frac{\gamma}2})$ terms
can be sharpened to an asymptotic series.  The functions $H_\gamma$ and all their $x$  derivatives are nonnegative.
Derivatives with respect to $x$ and $\gamma$ can be handled as above.

\subsection{Wronskian determinant}
The three functions $\Ai_\gamma$, $\Gi_\gamma$ and $\Hi_\gamma$
satisfy the same differential equation.  Here we will collect
properties of the Wronskian matrix defined by those functions.

The Wronskian determinant $W$ is independent of $x$ since there is no second derivative in the ODE  and we evaluate it
at $x=0$
\[
\begin{split}  
W = &
\det \left( \begin{matrix} 
 \Ai_{\gamma}(0) & \Gi_{\gamma}(0) & \Hi_\gamma(0) \\
-\Ai_{\gamma+1}(0) & - \Gi_{\gamma+1}(0)& \Hi_{\gamma+1}(0) \\
\Ai_{\gamma+2}(0) & \Gi_{\gamma+2}(0) & \Hi_{\gamma+2}(0) 
\end{matrix} 
\right) 
\\ = &  \pi^{-3} 3^{\gamma-1} 
\Gamma( \frac{\gamma+1}3) \Gamma( \frac{\gamma+2}3) \Gamma(\frac{\gamma+3}3) 
\det \left( \begin{matrix} 
  \sin( \frac{(\gamma-2)\pi}3 )   &   \cos( \frac{(\gamma-2)\pi}3 )   & 
     1  \\
-\sin(\frac{(\gamma-1)\pi}3)  & -\cos(\frac{(\gamma-1)\pi}3)  & 1 \\
\sin (\frac{\gamma \pi}3 )  & \cos( \frac{\gamma \pi}3 )  & 1 
\end{matrix} 
\right) 
\end{split} 
\] 
The Gaussian multiplication formula simplifies the product of the
$\Gamma$ functions
\[ \Gamma( \frac{\gamma+1}3) \Gamma( \frac{\gamma+2}3) \Gamma(\frac{\gamma+3}3)
=  2\pi  3^{-1/2- \gamma} \Gamma(1+\gamma).
\]

The remaining determinant can be expanded and simplified via addition theorems 
and evaluates to $3\sqrt{3}/2$.

Altogether, we arrive at 
\begin{equation} \label{wronski} 
 W =  \frac{\Gamma(\gamma+1)}{\pi^2}.  
\end{equation}  
In particular, the functions $\Ai_{\gamma}$, $\Gi_\gamma$ and $\Hi_\gamma$ 
are a fundamental system for the differential equation
\begin{equation}  (1+\gamma) u + x u_x - u_{xxx} = 0. \end{equation}

\subsection{Subdeterminants} 
Let 
\[
 f(x) = \Ai_{\gamma} \Gi_{\gamma}' - \Ai_\gamma' \Gi_\gamma 
= -\Ai_\gamma \Gi_{\gamma+1} + \Ai_{\gamma+1} \Gi_\gamma =: [\Ai_\gamma, \Gi_\gamma] 
\]
 and calculate
\[
\begin{split} 
  x f' - f''' = & x \Big(\Ai_{\gamma} \Gi_{\gamma+2} - \Ai_{\gamma+2}
  \Gi_\gamma \Big)\\ &
  - \Ai_{\gamma} \Gi_{\gamma+4} - 2 \Ai_{\gamma+1} \Gi_{\gamma+3} 
+ 2 \Ai_{\gamma+3} \Gi_{\gamma+1} + \Ai_{\gamma+4} \Gi_\gamma \\
  = & \Ai_\gamma ( x \Gi_{\gamma+2} - \Gi_{\gamma+4}) - \Gi_\gamma (x
  \Ai_{\gamma+2} - \Ai_{\gamma+4} ) \\ & +2 \Ai_{\gamma+1} ( x
  \Gi_{\gamma+1} - \Ai_{\gamma+3}) - 2 \Gi_{\gamma+1}( x
  \Ai_{\gamma+1} - \Ai_{\gamma+3} )
  \\
  = & (2+\gamma) \left(  \Ai_{\gamma} \Gi_{\gamma+1} -
 \Ai_{\gamma+1} \Gi_\gamma \right)  + 2(1+\gamma) 
\left( \Ai_{\gamma+1} \Gi_{\gamma} - \Ai_{\gamma} \Gi_{\gamma+1}\right)  \\
  = &  -\gamma \left( \Ai_{\gamma} \Gi_{\gamma+1} - 
\Ai_{\gamma+1} \Gi_\gamma \right) \\
  = & -(1+\tilde \gamma) f
\end{split} 
\]
with 
\[ \tilde \gamma = -1-\gamma. \]
Hence 
\[ f = c_1 \Ai_{\tilde \gamma} + c_2 \Gi_{\tilde \gamma} + c_3 \Hi_{\tilde \gamma}. \]
The function $f$ heritates the faster than polynomial decay for $ x>> 1$  from 
$\Ai_{\gamma}$. Thus $c_2 = c_3= 0$. The leading term to the right is 
\[ 
 \frac1{2\sqrt{\pi}} x^{\frac14+\frac\gamma2}  
  \frac{\Gamma(1+\gamma)}{\pi} x^{-1-\gamma} e^{-\frac23 x^{\frac32}} 
\]
We compare this with the asymptotic of $\Ai_{-1-\gamma}$ which gives 
\begin{equation}\label{AiGi} 
 [\Ai_\gamma, \Gi_\gamma] = 
\frac{\Gamma(1+\gamma)}{\pi} \Ai_{-1-\gamma} 
\end{equation}

Similarly 
\[ [\Ai_\gamma, \Hi_{\gamma}] = c_1 \Ai_{\tilde \gamma} 
+ c_2 \Gi_{\tilde \gamma} + c_3 \Hi_{\tilde \gamma} 
\] 
The leading term for $ x>> 1$  is 
\[   \frac1{\pi }   x^{\gamma}   \]
and hence 
\[ c_3=0, \qquad  c_2= \frac1{\Gamma(-\gamma)}. \] 
We recall that 
\[ \Gamma(1-s) \Gamma(s) = \frac\pi{\sin (s\pi)}. \]
to rewrite 
\[ c_2 = -\frac{\Gamma(1+\gamma)}\pi \sin(\gamma \pi) \]

The leading term for  $x << -1  $  is 
\[ \frac{\Gamma(1+\gamma)}{2 \pi^{3/2}} |x|^{-\frac34-\frac{\gamma}2} \cos\left( 
\frac23 |x|^{3/2} -\frac\pi4 -\frac{(\gamma+1)\pi}2 \right)
 \]
where 
\[
\begin{split}
  \cos\left( \frac23 |x|^{3/2} -\frac\pi4 -\frac{(\gamma+1)\pi}2
  \right) = & \cos\left( \frac23 |x|^{3/2}-\frac\pi4 -\frac{\tilde
      \gamma \pi}2 \right) \cos ((\gamma+1)\pi) 
\\ & + \sin\left(
    \frac23 |x|^{3/2}-\frac\pi4 -\frac{\tilde \gamma \pi}2 \right)
  \sin ((\gamma+1)\pi)
\end{split} 
\]

hence 
\[  c_1 = -\frac{\Gamma(1+\gamma)}{\pi} \cos( \gamma \pi ) \]
\begin{equation} \label{AiHi} 
[\Ai_{\gamma}, \Hi_\gamma] = - \frac{\Gamma(1+\gamma)}{\pi} \Big(\cos (\gamma \pi )
  \Ai_{\tilde \gamma} + \sin (\gamma \pi) \Gi_{\tilde \gamma} \Big).
\end{equation}

Finally 
 \[ [\Gi_\gamma ,\Hi_{\gamma}] = c_1 \Ai_{\tilde \gamma} 
+ c_2 \Gi_{\tilde \gamma} + c_3 \Hi_{\tilde \gamma}. 
\] 
The leading term for $ x>> 1$ is  
\[
 \frac{\Gamma(1+\gamma)} {\pi^{3/2}} |x|^{-\frac34-\frac{\gamma}2}
e^{\frac23 x^{\frac32} } \] hence
\begin{equation}  c_3 =  \frac{\Gamma(1+\gamma)}{\pi}.  \end{equation}
The leading oscillatory term for $ x << -1$  is 
\[ \frac{\Gamma(1+\gamma)}{2 \pi^{3/2}} |x|^{-\frac34-\frac{\gamma}2} 
\sin( \frac23 |x|^{3/2} -\frac\pi4 -\frac{(\gamma+1)\pi}2 ) 
\]
where 
\[\begin{split} 
 \sin( \frac23 |x|^{3/2} -\frac\pi4 -\frac{(\gamma+1)\pi}2 ) 
  = & \sin(\frac23 |x|^{3/2} -\frac\pi4- \frac{\tilde \gamma \pi}2)
  \cos ((\gamma+1) \pi) \\ & - \cos(\frac23 |x|^{3/2} -\frac\pi4-
  \frac{\tilde \gamma \pi}2) \sin ((\gamma+1) \pi)
\end{split} 
\]
hence 
\[ c_1 =  \frac{\Gamma(1+\gamma)}\pi \sin (\gamma\pi), \qquad c_2 =
-\frac{\Gamma(1+\gamma)}\pi \cos (\gamma \pi) \] 
and
\begin{equation} \label{GiHi} 
 [\Gi_\gamma,  \Hi_\gamma] = 
   \frac{\Gamma(1+\gamma)}{\pi} \Big(  \sin (\gamma \pi) \Ai_{\tilde \gamma}
 - \cos (\gamma\pi) \Gi_{\tilde \gamma} + 
\Hi_{\tilde \gamma} \Big) 
\end{equation}
We collect all the formulas in a proposition. 

\begin{prop} \label{minors}
 The following identities hold
\begin{equation} 
\begin{split} 
 [\Ai_\gamma,\Gi_\gamma] = & \frac{\Gamma(1+\gamma)}{\pi} \Ai_{-1-\gamma} \\
[\Ai_\gamma, \Hi_\gamma] = & \frac{\Gamma(1+\gamma)}{\pi}\left(-\cos(\pi \gamma) \Ai_{-1-\gamma} - \sin(\pi \gamma)  \Gi_{-1-\gamma}\right)
\\ 
[\Gi_\gamma,\Hi_\gamma] = & \frac{\Gamma(1+\gamma)}{\pi}\left( \sin(\pi \gamma) \Ai_{-1-\gamma} - \cos(\pi\gamma) \Gi_{-1-\gamma} + \Hi_{-1-\gamma}\right).   
\end{split}
 \end{equation}

\end{prop}

\section{Green's functions} 
\label{sec:fundamental} 

\subsection{The Green's function for \eqref{nonrescaled}}

 We consider the linear problem 
\begin{equation}\label{nonrescaled}   L_{\gamma} u := (1+\gamma)u + x u_x - u_{xxx} = f. \end{equation}  
The identities of Propositon \ref{minors} and \eqref{wronski} imply explicit formulas for Greens functions in terms of generalized 
Airy and Scorer's functions.  
There is a unique right inverse $T_L$ with integral kernel $K_L(x,y)$ supported on the left
of the diagonal.  It is for $x\ge y$
\begin{equation}   
\begin{split} 
\frac{K_\gamma^L (x,y)}{\pi}  = & \frac\pi{\Gamma(\gamma+1)} \Big\{   [\Ai_\gamma, \Gi_\gamma](y) \Hi_{\gamma}(x) 
\\ & + [ \Gi_\gamma, \Hi_\gamma](y) \Ai_{\gamma}(x) + [ \Hi_{\gamma}, \Ai_\gamma](y)] \Gi_\gamma(x) \Big\}
\\[3mm] = &  \Hi_{-1-\gamma} (y)\Ai_{\gamma}(x)+  \Ai_{-1-\gamma}(y) \Hi_\gamma(x) 
\\ & + \sin(\gamma \pi) \Big(\Gi_{-1-\gamma} (y) \Gi_\gamma(x)  +    \Ai_{-1-\gamma}(y) \Ai_\gamma(x)\Big)
\\ &  +\cos(\gamma \pi)\Big(  \Ai_{-1-\gamma} (y) \Gi_\gamma(x) -  \Gi_{-1-\gamma}(y) \Ai_\gamma(x) \Big).
\end{split}
\end{equation} 
      
It is easy to read off the leading terms of $K_\gamma^L$ in various asymptotic regimes.
Let  $x,y >> 1$. 
The leading term of the second line is given by the product of the $\Gi$ functions. It is 
\begin{equation}   \sin (\pi \gamma)  |x|^{-1-\gamma} |y|^{\gamma}. \end{equation} 
The third line decays fast as $x \sim y\to \infty$.

For $ x,y << 0 $ the only polynomial term without oscillations comes from the second line. It is 
\begin{equation}  -\sin \pi \gamma \cos^2( \pi\gamma)  |x|^{-1-\gamma} |y|^{\gamma},  \end{equation}

We recall that we will set $\gamma = \frac2p-1$ when we construct selfsimilar 
solutions,  and we will search solutions of finite energy, i.e. with 
$u_x \in L^2$ and $u \in L^{p+2}$. Let $X_0\subset C^1$ be the Banach space of functions such that the norm 
\begin{equation}  \Vert u \Vert_{X_0} =  \sup |(1+|x|)^{1+\gamma}  u| + |(1+|x|)^{2+\gamma} u_x|  \end{equation} 

The decay of the generalized Airy functions and of Scorer's function determine uniquely the right inverse 
which maps compactly supported functions to $X_0$.

\begin{theorem} \label{kernelorig} 
Let $-1 < \gamma < 0$. Then 
 there exists a unique right inverse $T_\gamma^a: C_{0}(\mathbb{R}) \to X  $ with the integral kernel  
\begin{equation}\label{kernel}   \begin{split} K_\gamma(x,y) = &  \pi \Big( \Hi_{-1-\gamma}(y) \Ai_\gamma(x) \chi_{y<x}- \Ai_{-1-\gamma}  (y) \Hi_\gamma (x)\chi_{x<y}    \Big)  \\
 & + \pi \sin (\gamma \pi)\Big(  \Gi_{-1-\gamma}(y) \Gi_{\gamma}(x)+\Ai_{-1-\gamma} (y) \Ai_{\gamma}(x) \Big)\chi_{x>y}  \\
 & + \pi  \cos(\gamma \pi) \Big(\Ai_{-1-\gamma} (y) \Gi_\gamma(x) - \Gi_{-1-\gamma}(y) \Ai_\gamma(x)\Big) \chi_{x>y}
       \end{split}
\end{equation} 
Moreover $T_\gamma$ maps $C_0$ to $X_0$. 
\end{theorem}

\subsection{The change of coordinates}

We will use the Green's function for the transformed problem. 
The equations 
\[ (1+\gamma)  v + y v_y - v_{yyy} = f \]
and 
\[ a( (1+\gamma)  u + xu_x) - u_{xxx} + u_x =  g \]
are equivalent via
\begin{equation}  
    x =     a^{-1/3} y-a^{-1}, \quad v(y) =  a u(a^{-1/3}y-a^{-1})   \quad        f(y) =     g(x), \qquad       
\end{equation} 
Then,  
\[ \begin{split} 
    u(x) &=  a^{-1}  v( a^{1/3} (x+ a^{-1}))\\
 &=  a^{-1}  \int K_\gamma(a^{1/3} (x+ a^{-1}), z)f(z) dz \\
&=  a^{-1} \int K_\gamma(a^{1/3} (x+ a^{-1}), z)g(a^{-1/3}z-a^{-1}) dz\\
 &=  a^{-2/3} \int K_\gamma(a^{1/3} (x+ a^{-1}), a^{1/3} (y + a^{-1} )) g(y)\, dy
\end{split} 
\]
Thus  
\[ u(x) = \int \tilde K^{a}_\gamma(x,y) g(y) dy \]
where
\[
 \tilde K^a(x,y) =    a^{-2/3} K_\gamma ( a^{1/3}(x + a^{-1}), a^{1/3} (y +a^{-1}))
 \]
 We apply it to $g = \partial_x F$, where one  integration by parts yields
\begin{equation}   u(x) = -\int \partial_y \tilde K^{a}(x,y) F (y)  dy = \int K^a(x,y) F (y)  dy =: T^a_\gamma F\end{equation} 
with a new kernel
\begin{equation}  \begin{split} \frac{a^{1/3}}{\pi} K^a(x,y) = &  -\Ai_{-\gamma} (a^{1/3}(y+ a^{-1})) \Hi_\gamma(a^{1/3}(x+ a^{-1}) \chi_{x<y}
 \\ & \qquad  -   \Hi_{-\gamma}( a^{1/3} (y + a^{-1})) \Ai_{\gamma}( a^{1/3} (x + a^{-1})) \chi_{y<x} 
\\ & + \sin(\gamma \pi)  \Big( \Gi_{-\gamma} (a^{1/3} (y + a^{-1})) \Gi_{\gamma}(a^{-1/3} (x+ a^{-1}) )
\\ & \qquad + \Ai_{-\gamma} (a^{1/3} (y + a^{-1})) \Ai_{\gamma}(a^{1/3} (y+ a^{-1}) )\Big)\chi_{x > y}
\\  & +\cos(\gamma\pi) \Big( \Ai_{-\gamma} (a^{1/3} (y + a^{-1})) \Gi_{\gamma}(a^{-1/3} (x+ a^{-1}) )
\\ & \qquad - \Gi_{-\gamma} (a^{1/3} (y + a^{-1})) \Ai_{\gamma}(a^{1/3} (y+ a^{-1}) )\Big)\chi_{x > y.}
\end{split} \label{transkernel} 
 \end{equation} 

We arrive at the reformulation 
\begin{equation}    u(x) + T^a_\gamma  (|u|^p u - \langle u,Q_x \rangle  Q_x)  = 0 \end{equation} 
of the bifurcation problem \eqref{bifurcationsigma}.

\subsection{Dependence on $a$ and $\gamma$} 
The previous considerations show that

\begin{equation}  
\begin{split} 
\frac{a^{2/3}K_L (x,y)}{\pi}  =  &  \Hi_{-1-\gamma} (a^{-2/3}(1+ay))\Ai_{\gamma}(a^{-2/3}(1+ax))\\ &\  -  \Ai_{-1-\gamma}(a^{-2/3}(1+ay)) \Hi_\gamma(a^{-2/3}(1+ax)) 
\\ & + \sin(\gamma \pi) \Big(\Gi_{-1-\gamma} (a^{-2/3}(1+ay)) \Gi_\gamma(a^{-2/3}(1+ax))  \\ & \ +    \Ai_{-1-\gamma}(a^{-2/3}(1+ay)) \Ai_\gamma(a^{-2/3}(1+ax))\Big)
\\ &  +\cos(\gamma \pi)\Big(  \Ai_{-1-\gamma} (a^{-2/3}(1+ay)) \Gi_\gamma(a^{-2/3}(1+ax)) 
\\ & \ -  \Gi_{-1-\gamma}(a^{-2/3}(1+ay)) \Ai_\gamma(a^{-2/3}(1+ax)) \Big).
\end{split}
\end{equation}
is the forward Green's function. Given $y$ it is a solution to the homogeneous 
differential equation with initial condition 
\[  u(y) = u'(y) = 0, \qquad u''(0) = 1 \]
It depends analytically on $x$, $y$, $ a \in \mathbb{R}$ and $\gamma$ away from the diagonal $x=y$.
We claim that  $K^a$ is smooth with respect to $a$, $ \gamma$, $x$ and $y$. To see this we have to show that 
\[ a^{-1/3} \Ai_{\gamma}(a^{-2/3}(1+ax) \Hi_{-\gamma}(a^{-2/3}(1+ay) \]
is smooth  in $a$ and $\gamma$. It suffices to consider this at $x=y=0$, since 
solutions to analytic ODEs are analytic. We claim that 
 \[ a^{-1/3} \Ai_{\gamma}(a^{-2/3}) \Hi_{-\gamma}(a^{-2/3}) \]
is smooth  with respect to $a\in \R$. Analyticity with respect to $\gamma$ follows from 
analyticity of $\Ai_\gamma$ and $\Hi_{-\gamma}$ for fixed $a$. Smoothness in $a$ 
and even analyticity is obvious for $a \ne 0$. At $a=0$ smoothness follows from the 
asymptotics of $\Ai_\gamma$ and $\Hi_{-\gamma}$ in \eqref{airyasym} and \eqref{hiasym}.

The following Lemma  quantifies the dependence on $a$ in a crucial region. 
It is an immediate consequence of the asymptotics of the Airy and Scorers functions.
\begin{lemma}\label{dependona}  The following estimate 
\[ 
\begin{split} 
  \left| K^a(x,y) - \left( e^{-\frac12 |x-y|} + a \chi_{x>y} (1+ax)^{-1-\gamma} (1+ay)^{-1+\gamma}   \right)\right| &\\ & \hspace{-5cm}    \le c_\delta  \left(  a^2+ |a| 
e^{-\frac12    |x-y|} \right) 
\end{split} 
 \]
holds for $|x|,|y| \le a^{-1/2}$ . 
\end{lemma} 
\begin{proof} First we observe 
 \[ \begin{split} \left| a^{-1/3} \Gi_\gamma(a^{-2/3}(1+ax)) \Gi_{-\gamma}(a^{-2/3}(1+ay)
- a(1+ax)^{-1-\gamma} (1+ay)^{-1+\gamma}  \right|  & \\ 
& \hspace{-6cm}  \le a^3 ((1+ax)^{-3} 
\end{split} 
\]
The terms $\Ai_{-\gamma} \Gi_\gamma$, $\Gi_{-\gamma} \Ai_\gamma$ and 
$\Ai_{-\gamma} \Ai_\gamma$ are much smaller. To be precise we assume $x \le y$ and
estimate 
\[
\begin{split} 
 \pi^{-1} a^{-1/3}\left| \Ai_{-\gamma}(a^{-2/3}(1+ay) ) \Hi_\gamma(a^{-2/3}(1+ax)e^{y-x} -\frac12 \right|  
 \hspace{- 8cm} & 
\\ = &  (1+ax)^{-\frac14+\frac{\gamma}2} (1+ay)^{-\frac14-\frac{\gamma}2} 
  \left|\frac12  e^{\frac{2}{3a}  ( (1+ax)^{\frac32} -(1+ay)^{\frac32})  +(y-x)}   -\frac12 \right| 
\\ \le & c \left| e^{ \frac{a}4  (1+ax)^{-\frac12} (y-x)^2} -1 \right| 
\\ \le & c a |x-y|^2 
\end{split} 
\]
The case $x \le y$ is similar. 
  \end{proof} 

\section{The implicit function theorem} 
\label{sec:weighted} 
\subsection{The operator $T^a_\gamma$ in weighted function spaces}

We rewrite the problem as a fixed point problem for the identity plus a compact 
map.  Then the Fredholm alternative will allow us to apply the implicit function 
theorem. Things however are not as simple as they may appear from this 
description: The derivatives with respect to $a$ and $\mu$ are not bounded 
in this functional analytic setting. They have to be handled by different arguments 
in the next section.

The following result is the basic linear estimate for the operator $T^a_\gamma$.
It is a weighted estimate with a weight tailored for the problem at hand. This is necessarily involved. 

The asymptotics on the left is essentially given by 
\[ \Hi_{\gamma}(a^{-2/3}(1+ax)) / \Hi_\gamma(a^{-2/3})\]
and on the right by 
\[ \Gi_{\gamma}(a^{-2/3}(1+ax)) / \Gi_\gamma(a^{-2/3}). \]
This decay is to a certain extent captured by the weights below.

\begin{prop} \label{keyest} 
Let $k \ge 0$.  
There exists $c>0$ such that the following is true.
Let   $0< a \le 1/10 $,    $|\gamma+\frac12| <  \frac18$ and
\begin{equation}\label{wa}     w^a(x):= \left\{ \begin{array}{cl}     e^{-\frac{1 }{3 a}} (1+a^{-2/3}|1+ax|)^{-\frac38}  & \text{ if } x \le -a^{-1} \\
                                            \exp(\frac{1}{3a} [(1+ax)^{3/2}-1])  & \text{ if }  -a^{-1} \le x \le 0 \\
                         (1+x)^k (1+ax)^{-1-\gamma-k} & \text{ if } x \ge 0 
                      \end{array} \right. 
\end{equation} 
and
 \begin{equation} \label{wai} 
  w_i^a(x) =  \left\{ \begin{array}{cl} 
 e^{-\frac{1}{2 a}}(1+  a^{-2/3}|1+ ax|)^{-\frac32}  & \text{ if } x \le -a^{-1} \\
 \exp(\frac{1}{2a} [(1+ax)^{\frac32}-1]))  & \text{ if }  -a^{-1} \le x \le 0 \\
                      (1+x)^k      (1+ax)^{-1-\gamma-k}  & \text{ if } x \ge 0 
\end{array} \right. 
  \end{equation} 
and 
\begin{equation} 
 w^0 = \left\{ \begin{array}{rl} e^{- |x|/2} & \text{ if } x < 0 \\ 
   1 & \text{ otherwise }  \end{array} \right. \qquad  w^0_i=\left\{ \begin{array}{rl} e^{- 3|x|/4} & \text{ if } x < 0 \\ 
   1 & \text{ otherwise }  \end{array} \right.
\end{equation} 
Then 
\begin{equation}  \sup_{x,0\le a \le \frac12, |\frac12+\gamma|\le \frac18 } |T^a_\gamma f(x)|/w^a(x) \le c \sup_x  |f(x)|/w_i^a(x). \end{equation} 
 \end{prop} 

 The complexity of the weight 
reflects the different asymptotic areas, and the proof consists in 
decomposing operator and domain in smaller pieces for which elementary 
estimates become possible. The proposition is an immediate consequence 
of 

\begin{lemma} \label{kernelest} 
There exists $c>0$ independent of $x$, $a$ and $\gamma$ such that  
\[  \int |K^a_\gamma(x,y)| w_i^a (y) dy \le c w^a(x) 
\]
for $x \in \mathbb{R}$, $0\le a \le  1$ and $|\frac12 +\gamma| \le \frac18$. 
\end{lemma}

\begin{proof} 
We will restrict ourselves to $k=0$, with marginal differences for positive $k$.
We recall that 
\begin{equation}  |\Ai_\gamma(x) |+ |\Gi_\gamma(x)|+ |\Hi_\gamma(x)| \le c (1+|x|)^{-\frac38} \end{equation} 
for $ x\le  0$ and $|\gamma+\frac12| \le \frac18$.

\noindent{\bf Step 1: $x\ge 0$. }
There are contributions from the integrals over $(-\infty,-a^{-1})$, $(-a^{-1},0)$, 
$(0,x)$ and $(x,\infty)$. We deal with them in reverse order, and we begin 
 $a^{-1/3}\Ai_{-\gamma}(a^{-2/3})(1+ay)) \Hi_{\gamma}(a^{-2/3}(1+ax))$ - the kernel
for $y>x$.  
Then
\[
\begin{split} 
 a^{-1/3} &  (1+ax)^{1+\gamma} \Hi_\gamma(a^{-2/3}(1+ax))  
 \int_{x}^\infty \Ai_{-\gamma}(a^{-2/3}(1+ay)) (1+ay)^{-1-\gamma}       dy  
\\ &  \lesssim   \int_x^\infty   e^{\frac2{3a}\Big((1+ax)^{\frac32} -(1+ay)^{\frac32}\Big)}\left( \frac{1+ax}{ 1+ay}\right) ^{1+\frac\gamma2} [(1+ax)(1+ay)]^{-\frac14}  dy  
\\ & \lesssim 
 \int_x^\infty   e^{(1+ax)^{1/2}(x-y)}   \left( \frac{1+ax}{ 1+ay}\right) ^{\frac54+\frac\gamma2} (1+ax)^{-\frac12}  dy  
\\ & \lesssim  C 
\end{split} 
\]
holds uniformly in $x \ge 0$ if $|\gamma+\frac12| \le \frac18 $.
Similarly 
 \[ 
 a^{-1/3}    \Ai_\gamma(a^{-2/3}(1+ax))  
 \int_{0}^x \Hi_{-\gamma}(a^{-2/3}(1+ay)) \left( \frac{1+ax}{1+ay}\right)^{1+\gamma}       dy  \le c. 
\]
in the same range. 

Next we consider the contribution of the product of the functions $\Gi$,  using $ \Gi_{\gamma}(a^{-2/3}(1+ax)) \sim  a^{\frac23(1+\gamma)}   (1+ax)^{-1-\gamma} $ for $x >0$:
\[  \begin{split} 
a^{-1/3} &  \int_0^x \left(\frac{1+ax}{1+ay}\right)^{1+\gamma} \Gi_\gamma(a^{-2/3}(1+ax)) 
\Gi_{-\gamma}(a^{-2/3}(1+ay)) dy 
 \\ \lesssim &   \int_0^x a (1+ay)^{-2} dy 
\\ \lesssim  & 1. 
\end{split} 
\] 

 The products 
\[ a^{-1/3}\Ai_{-\gamma}(a^{-2/3}(1+ay) )\Gi_{ \gamma}(a^{-2/3}(1+ax)) \]
and 
\[
a^{-1/3} \Gi_{-\gamma}(a^{-2/3}(1+ay)) \Ai_{\gamma}(a^{-2/3}(1+ax)) \]
are much smaller.

We observe that 
\[  \Ai_\gamma(a^{-2/3}(1+ax))\lesssim \Ai_{\gamma}(a^{-2/3})  w^a(x),  \]
\[  \Gi_{\gamma}(a^{-2/3}(1+ax)\lesssim \Gi_{\gamma}(a^{-2/3})  w^a(x) \]
for $x \ge 0$ and it suffices to bound the  contribution from $y \le 0$ at $x=0$ 
to get the same bound for all nonnegative $x$.

The estimates 
\[
\int_{-\infty}^0 \Hi_{-\gamma}(a^{-2/3}(1+ay)) w^{a}_i(y)dy \lesssim  \Hi_{-\gamma}(a^{-2/3}),
 \]
\begin{equation} \label{polleft} 
  e^{-\frac{1}{2a} } \int_{-\infty}^{-a^{-1}} (1+a^{-2/3}(1+ay))^{ \frac1{16}-\frac32}
  dy \lesssim   e^{-\frac1{3a} }  
\end{equation} 
and
\[
\int_{-a^{-1}}^0 (1+ a^{-2/3}(1+ay))^{-1-\gamma}  e^{ \frac1{2a}[( 1+ a^{-2/3}(1+ay))^{\frac32}-1]}  dy 
\lesssim a^{\frac23(1+\gamma) } \sim \Gi_\gamma(a^{-2/3}) 
\]
are straight forward.  This completes the estimate for $x>0$.

\medskip 

\noindent{\bf Step 2: $x<0$.} In view of the first  substep above (with $x=0$) the contribution from $y>0$ 
is controlled by  the obvious estimate 
\[ \Hi_{\gamma}(a^{-2/3}(1+ax))\lesssim  H_\gamma(a^{-2/3}) w^a(x). \]
We consider   the contribution from $y \le 0$ to $x \in [-a^{-1}, 0]$.
There are contributions from three different intervals: $(-\infty,-a^{-1})$, $(-a^{-1},x)$, and $(x,0)$, which we consider step by step. We consider first 
 the product of $\Ai$ and $\Hi$. The desired estimate is  
\[
\begin{split} 
\sup_{-a^{-1} \le x \le 0}   a^{-1/3} \int_{-a^{-1}}^0   \frac{( 1+ a^{-2/3}(1+ay))^{\gamma-1} }{(1+a^{-2/3}(1+ax))^{\gamma+1} }  \times & \\ & \hspace{-5cm} 
   \times  e^{-\frac2{3a} |(1+ax)^{\frac32} - (1+ay)^{\frac32}| }
 e^{-\frac1{2a}[ (1+ay)^{\frac32}-1]} e^{\frac{1}{3a}[(1-ax)^{\frac32}-1]} 
dy
\end{split} 
\]
which is trivial once broken up into different cases: $ x=0$, $y \le
x$, $ -\frac12a^{-1}\le x< y$ and $-a^{-1} \le x \le -\frac1{2a}$.
The contributions from the other terms in the Greens function are much
smaller.  Finally the contribution (to $x\in [-a^{-1},0]$) from $y \le
-a^{-1}$ is controlled by \eqref{polleft}.

\noindent {\bf Step 3: The case $x< -a^{-1}$, contribution from $y\le 0$.} 
 Again we have to  consider  the integrals over $(-\infty,x)$, $(x,a^{-1})$ 
and $(a^{-1},0)$. The integral over $(a^{-1},0)$ has been evaluated above. 
The obvious estimates 
\[   |\Ai_\gamma(a^{-2/3}(1+ax))|+ |\Gi_\gamma(a^{-2/3}(1+ax))|+ \Hi_\gamma (a^{-2/3}(1+ax))   \lesssim \frac{w^a(x)}{ w^a(-a^{-1})} \]
complete that part. 

The kernel satisfies 
\[    |K^a_\gamma(x,y)| \lesssim  a^{-1/3} (1+a^{-2/3}(1+ax))^{-\frac38} (1+a^{-2/3}(1+ay))^{\frac1{16} }  \]
for $x,y \le -a^{-1}$.  Now

\[  \int_{-\infty}^{a^{-1}}  (1+a^{-2/3}(1+ay))^{-\frac23 +\frac18}   dy \lesssim     
 a^{-1/3}  \]
completes the proof .
  \end{proof} 

We reformulate the bifurcation problem as a 
fixed point problem
\begin{equation} \label{fixedpoint} 
  u(x) =  T^{a}_\gamma \left(|u|^p u - \langle u, Q_x \rangle Q_x\right) 
\end{equation}
where we search $u$ in a neighborhood of $Q$.

We introduce $v =  u/w^a$ with $w^a$ from \eqref{wa} and rewrite the problem as 
\[ F(v)= 0 \]
with 
\[ F(v) =  v -  (w^a)^{-1}  T^a_\gamma \left[  |vw^a|^p v w^a   - \langle v w^a, Q_x\rangle Q_x \right]. \] 
Proposition \ref{keyest} implies that the map is $j$ times
Frechet differentiable on the space of bounded continuous functions,
for every nonnegative integer $j \le 3$.

We turn to the study of the linearization of \eqref{fixedpoint}  at $Q$. We include the weights into the operator and consider 
\[       \tilde T(a,p,\gamma, v)  u := (w^a)^{-1} T^a\left[(p+1)|vw^a|^p w^a u 
+ \langle u, w^a Q_x \rangle Q_x \right]. \]  
Let  $C_b(\mathbb{R})$  denote the space of continuous functions equipped by the supremums norm, 
and the closed subspace of functions with limit $0$ as 
$ x\to \pm \infty$ by $C_0$.
The  space of linear operators from the normed space $X$ 
to the normed space $Y$ is denoted  by $L(X,Y)$, which we equip with the operator norm.

\begin{cor}\label{continuity}  
The map 
\[ \begin{split} \Big( [0,1] \times \Big[-\frac12-\frac18,- \frac12+\frac18\Big]\times [3,q)\times C_0(\R) \Big) \to  & L (C_b,C_b) \\
(a,\gamma,p,v) \to & \left( u \to   \tilde T(a,p,\gamma,v)  u \right)  
\end{split} 
\]
is continuous. 
\end{cor} 
\begin{proof} 
The map 
\[   (a ,p) \to   Q_x/w^a_i  \in L^1  \]
is clearly continuous as is 
\[  (v,u,a,p) \to  |vw^a|^p w^a u/w^a_i.  \]
This implies continuity with respect to $v$, uniform with respect to $a$ and $p$. 
Hence it suffices to prove the continuity for the composition with the multiplication 
by a characteristic function, 
\[ (a,\gamma,p,v) \to  \left( u \to \tilde T(a,p,\gamma,v) \chi_{[-R,R]} u \right). \]
 The proof of Proposition \ref{keyest} implies that 
\[  x \to \tilde T(a,p,\gamma,v) \chi_{[-R,R]} u(x)  \]
converges to zero as $x \to -\infty$, uniformly for bounded $v$ and $u$ and $a$ and 
$p$ as in the theorem. 
Continuity with respect to $p$ is obvious. On the right hand side the situation is slighty different: 
Since we apply the operator to a function with compact support, the only terms which does not decay as $x \to \infty$ comes from $\Gi_\gamma(a^{-2/3}(1+ax))$. 
This term is clearly continuous with respect to $a$ and $\gamma$.
Continuity with respect to $\gamma$ and $a$  follows from the continuity of the Airy and Scorer  functions,  their asymptotics and the continuity of the Green's function.  
 \end{proof}

The invertibility of the linearization in a neighborhood of the bifurcation point is contained in the next proposition.  We denote 

\[  S^a_v u  =  u-    w_a^{-1} T^a_\gamma \big[  (p+1) w_a^{p+1} v^p u - \langle u, w_a Q_x\rangle Q_x \big] \]
 
\begin{prop}\label{invertibility}  
There exists $\delta >0$ such that 
$ S^a_v  : C_{b} \to C_b$ is invertible with an inverse whose 
norm is uniformly bounded for 
\[ |\gamma+\frac12| \le \delta,0\le  a\le \delta \text{ and } \sup \frac{| v-Q|}{w^a}\le \delta. \]
\end{prop}
\begin{proof} 
The operator $S^a_v $    is bounded by Proposition \ref{keyest} and the norm
continuity at $v=Q/w^a$ is the content of Corollary \ref{continuity}.  It thus 
suffices to consider invertibility at $a=0$ and $v=Q/w^a$. Clearly 
\[  u \to (w^0)^{-1} T^0[(p+1) Q^p (w^0)^{p+1} u ] - \langle  u, w^0 Q_x\rangle Q_x ] \]
is compact. We recall that the integral kernel of $T^0$ is $\frac12 e^{-|x-y|}$.

By the Fredholm alternative $S^0_Q$ is invertible if the null space is
trivial. Consider
\begin{equation}   
  u = T^0_\gamma\left[ (p+1)Q^p u - \langle  u , Q_x\rangle Q_x\right].\label{homogeneous}  \end{equation} 
We claim that there is only the trivial bounded solution. 
Suppose  that $u$ satisfies the homogeneous equation \eqref{homogeneous}. 
Since the kernel decays fast also $u$ decays fast, and the same holds for the 
derivatives. Hence
\[ u - u_{xx} - (p+1) Q^pu + \langle u , Q_x \rangle Q_x = L u -
\langle u, Q_x \rangle Q_x = 0. \] 
We take the inner product with
$Q_x$. Then $\langle u, Q_x\rangle = 0 $ since $LQ_x=0$. The null
space of $L$ is spanned by $Q_x$ and hence $u = 0$. This null space is
trivial, by the Fredholm alternative $S^0_Q$ is invertible, and this
remains so in a small neighborhood of the coefficients and $Q/w^a$.
\end{proof}

We continue with an estimate which implies that $Q$ is almost a
solution to the fixed point problem.  This is important since $F$
fails to be differentiable with respect to $a$ and $\gamma$.
\begin{lemma}\label{starting} 
There exists $C>0$ such that 
\[   
\sup_{x} |(w^a)^{-1}(Q-T^{a}_\gamma Q^{p+1})|
\le C a. 
\] 
\end{lemma} 

\begin{proof}
We observe that 
\[ Q- T^0_\gamma \left[ Q^{p+1}+ \langle Q, Q_x\rangle Q_x\right] = 0 \]
since $Q$ satisfies the soliton equation. The assertion is equivalent to  
\[ \left|(T^0_\gamma-T^a_\gamma)Q^{p+1}\right| \le ca w^a(x) \] 
which
we address now.  Since $Q \lesssim e^{-|x|} $ there exists $c>0$
independent of $a$, $ \gamma$ and $p$
\[ \sup_{|x| \ge c |\ln a|}| (w^a_i)^{-1} Q^{p+1} |\le a. \] 
so that
with $\chi$ the characteristic function of the complement of $[-c |\ln
a|,c|\ln(a)|] $
\[ \sup_x (w^a(x))^{-1} |T^a_\gamma \chi Q^{p+1}(x)| \lesssim a. \] 
it suffices to verify 
 \[  \left|(T^0_\gamma-T^a_\gamma) \chi Q^{p+1}\right|  \lesssim  a w^a(x). \]
Checking the kernel  we see that 
\[   |T^0_\gamma \tilde \chi Q^{p+1}(x)  | + |T^a_\gamma \tilde \chi Q^{p+1} (x)| \lesssim  a w^a(x) \]
if $ |x| \ge 2 |\ln(a)|$. Now $|Q^p| \le e^{-p|x|}$, $Q^{s}$ is integrable whenever $s>0$   and $Q\lesssim w^a_i$. 
Thus the statement will follow from 
\begin{equation}  \sup_{|x|,|y| \le c |\ln a|}   e^{\frac23 \max\{-x,0\}} \left(K^a_\gamma(x,y)-\frac12 e^{-|x-y|} \right) e^{-3 |y|}  \lesssim a \end{equation}
which is a consequence of Lemma \ref{dependona}.
\end{proof}

\begin{prop} Let $q >4$ and $3 \le p \le q$. 
Then there exists $\varepsilon$ 
and $C>0$ so that there is a unique fixed point $u$ to 
\[  u = T^{a}_\gamma (|u|^p u-\langle u, Q_x\rangle Q_x ) \]
with 
\begin{equation}\label{lipschitza}   \sup_x  (w^a(x))^{-1}  |u(x) -Q(x) | + |\langle u, Q_x \rangle| \lesssim a \end{equation}
for
\[ \max\left\{ \left|\frac12+\gamma\right| , a \right\}  \le \varepsilon. \] 
The map $(a,\gamma,p) \to (w^a)^{-1} u \in C_b(\mathbb{R})$ is continuous. 
\end{prop}

\begin{proof} 
We write $u =  w^a v-Q$. Then we search a fixed point to 
\[ 
\begin{split} 
  v =& (w^a)^{-1}\left(T^a_\gamma ( |w^a v+Q|^{p}(w^a v+Q)- \langle v, Q_x\rangle Q_x) -  Q \right) \\ =& (w^a)^{-1} \left( T^a_\gamma( |w^a v+Q|^{p}(w^a v+Q) -Q^{p+1} -\langle v, Q_x \rangle Q_x ) \right) 
\\ & + (w^a)^{-1} ( T^a_\gamma Q^{p+1} -Q). 
\end{split} \]
The second term on the right hand side is bounded by a constant times $a$ by Lemma \ref{starting}. The  derivative at $v=0$ is invertible
by Lemma \ref{invertibility} with a uniformly bounded inverse.  The existence of a unique fixed point with the desired properties follows now by the implicit function theorem.
\end{proof}

\section{Asymptotics and differentiability} 
\label{sec:expansion}

In the last section we have constructed a unique fixed point $u$ 
 to 
\begin{equation}\label{firststep}   u = T^a_\gamma \Big(|u|^p u -  \langle u, Q_x \rangle Q_x\Big) 
\end{equation} 
with 
\[  \Vert (u-Q)/w^a \Vert_{sup} <<  1. \]
Moreover it satisfies  
\begin{equation}    |u-Q| \le c a w^a     \end{equation} 
with a constant which is uniform in $a$, $p$ and $\gamma$.
It follows immediately from the integral representation and the decay that 
$u_x$ is square integrable. Moreover $u/w^a$ depends continuously on $a$, $p$ and 
$\gamma$ considered as a map to $C_b(\mathbb{R})$. It remains to show 
that this map is smooth for every $x$, to give bounds for the derivatives, and
to prove the uniqueness statement. Here we turn to differentiability and bounds for 
the derivatives.

As a first step and a warm up we consider the simpler term first. This term
will not enter the asymptotics of the fixed point, but we need it 
to prove differentiability with respect to $a$.

\subsection{The asymptotics of $v^a_\gamma:= T^a_\gamma Q_x$} 

We define 
\begin{equation}  c_0=c_0(a,p,\gamma) = \pi   \int \Ai_{1-\gamma}(a^{-2/3}(1+ay)) Q(y) dy 
\end{equation} 
and 
\begin{equation} d_0 = d_0(a,p,\gamma) = \pi \int \Gi_{1-\gamma}(a^{-2/3}(1+ay)) Q(y) dy. 
\end{equation}

\begin{prop} \label{firstest} 
The following estimates hold for $\kappa <1$
\begin{equation} 
   e^{ -\kappa x}\left|\partial_x^j \partial_a^k \partial_p^l 
( T^a_\gamma Q_x + c_0 \Hi_{\gamma}(a^{-2/3}(1+ax))) \right| \le \frac{c(j,k,l)}{\sqrt{1-\kappa}}    
\end{equation}
if $ x \le 0$ and if $x\ge 0$
\begin{equation}  
  e^{ \kappa x}\left|\partial_x^j \partial_a^k \partial_p^l
( T^a_\gamma Q_x + d_0 \Gi_{\gamma}(a^{-2/3}(1+ax))) \right| \le \frac{c(j,k,l)}{\sqrt{1-\kappa}}.   
\end{equation} 
Moreover there are the asymptotic series
\[     c_0   H_\gamma(a^{-2/3}) = \sum_{k=0}^\infty  \alpha_k(\gamma,p)  a^k   \]
and 
\[    d_0 \Gi_\gamma(a^{-2/3}) = \sum_{k=2}^\infty  \beta_k(\gamma,p) a^k     \]
with nontrivial leading term $\beta_2$ resp $\alpha_0$.  The coefficients are 
smooth functions of $\gamma$ and $p$, with bounds depending only on $k$.
\end{prop} 

\begin{proof} 

We can define a  solution to the linear equation 
\begin{equation} \label{linearQ}  a(\gamma v + xv_x) - v_{xxx} +v_x = Q_{xx}  \end{equation} 
by an integral kernel  $K_L$, supported in $ y<x$, which  is given by (compare with Theorem \ref{kernelorig}) 
\begin{equation}   
\begin{split} 
\frac{a^{2/3}K_L (x,y)}{\pi}  =  &  \Hi_{-1-\gamma} (a^{-2/3}(1+ay))\Ai_{\gamma}(a^{-2/3}(1+ax))\\ &\  \qquad +  \Ai_{-1-\gamma}(a^{-2/3}(1+ay)) \Hi_\gamma(a^{-2/3}(1+ax)) 
\\ & + \sin(\gamma \pi) \Big(\Gi_{-1-\gamma} (a^{-2/3}(1+ay)) \Gi_\gamma(a^{-2/3}(1+ax))  \\ & \ \qquad +    \Ai_{-1-\gamma}(a^{-2/3}(1+ay)) \Ai_\gamma(a^{-2/3}(1+ax))\Big)
\\ &  +\cos(\gamma \pi)\Big(  \Ai_{-1-\gamma} (a^{-2/3}(1+ay)) \Gi_\gamma(a^{-2/3}(1+ax)) 
\\ & \qquad  -  \Gi_{-1-\gamma}(a^{-2/3}(1+ay)) \Ai_\gamma(a^{-2/3}(1+ax)) \Big).
\end{split}
\end{equation}
Two solutions to \eqref{linearQ} differ by a solution to the homogeneous problem. The formula  
\[    \int_{-\infty}^x K_L(x,y) Q_{xx}(y) dy  - c_0  \Hi_\gamma(a^{-2/3}(1+ax)) 
\]
defines a solution to \eqref{linearQ}
hence it differs from $v^a_\gamma$ by a solution to the homogeneous equation. 
Both functions and their derivatives are bounded by a multiple of $w^a$ for $x \le 0$, and hence their difference is   a multiple of $\Hi_\gamma$.  
But the coefficients of the leading term are the same because of the choice of  $c_0$, and hence both are the same. 

We have
\[  \partial_x^j \left[v-c_0 \Hi_\gamma(a^{-2/3}(1+ax)\right]   =  \int_{-\infty}^x \partial_x^j K_L(x,y) Q_{xx}(y) dy 
 \]
if $j=0,1,2$. For $j\ge 3$ there is an additional term consisting of a finite 
sum of $Q$ and its derivatives. The kernel obviously reproduces exponential decay
up to polynomial factors.  The leading contribution for $ \kappa \to 1$ comes from 
$x $ close to $1$.

We argue similarly for $x \ge 0$, but this time with the standard  kernel $K^a_\gamma$.
The leading part comes from  $a^{-1/3}\Gi_\gamma(a^{-2/3}(1+ax)\Gi_{-\gamma}(a^{-2/3}(1+ax))$.  
 The products of $\Ai_\gamma \Hi_{-\gamma} $ and $\Ai_{-\gamma} \Hi_\gamma$ and reproduce  exponential decay if $\kappa  <1$ and with a constant $\sqrt{1-\mu}$ if $\mu$ is close to $1$. The other components 
of the kernel are supported in $y\ge x$. 
They reproduce  the exponential decay of $Q_{xx}$.
 We verify the asymptotic formulas for $c_0$ and $d_0$ and we  begin with $d_0$. 
Let 
\[ m_k = \int x^k Q(x) dx. \]
be the moments of $Q$.  
They are  smooth functions of  $p$, and independent of $\gamma$ and $a$.
A Taylor expansion of $\Gi_{1-\gamma}$ gives  the asymptotic series
 \[ 
\begin{split} 
 \int \Gi_{1-\gamma}(a^{-2/3}(1+ay)) Q(y) dy 
\sim &       \sum_{k=0}^\infty \frac{1}{k!}   m_k  a^{k/3}   \Gi_{1+k-\gamma}(a^{-2/3}) 
\\ \sim &  a^{2-\frac{2\gamma}3} ( m_0 + (m_1+  m_0)  a + \dots  )  
\\ \sim &( \Gi_\gamma(a^{-2/3}))^{-1}  ( \sum_{k=2}^\infty  \beta_k a^k ) 
\end{split} 
\]
with $\beta_2  \ne 0$.

 Differentiability of the coefficients with respect to $p$ 
is obvious. Differentiability with respect to $\gamma$ follows from the 
differentiability of $\Gi_{1-\gamma}$ with respect to $\gamma$ and the corresponding bounds. 
The difference to a partial sum is easily controlled by the estimates for
$\Gi_\gamma$ and its derivatives.

For the expansion of $c_0$  we write 
\[
\begin{split} 
 c_0 = &  a^{-2/3} \int \Ai_{1-\gamma}(a^{-2/3}(1+ay))e^{y} (e^{-y} Q(y) ) dy 
  \\ = &- a^{-2/3} \int  \int_0^y\left[  \Ai_{1-\gamma}(a^{-2/3}(1+as))e^s\right] ds \frac{d}{dy} (e^{-y} Q(y)) dy. 
\end{split}
\] 
The function $\frac{d}{dy} (e^{-y } Q(y))$ is a Schwartz function with exponential decay.
The zeroth  moment is $-1$, but $\int_0^x \Ai_{1-\gamma}$ vanishes at 
$x=0$,   and the leading contribution  comes from the next term, 
\[ 
\begin{split} 
c_0 =& a^{-2/3}  \Ai_{1-\gamma}(a^{-2/3}) \int y \frac{d}{dy} (e^{-y}Q(y)) dy 
\\ & +\left( \frac14+\frac{\gamma}2\right)   (a^{1/3}  \Ai_{1-\gamma}(a^{-2/3}) ) 
\int y^2 \frac{d}{dy} e^{-y} Q(y) dy \dots
\\ = & (\Hi_\gamma(a^{-2/3}))^{-1} ( \sum_{j=0}^\infty  \alpha_j  a^j)  
\end{split}  
\]
with $\alpha_0 \ne 0$.  Any derivative on $ \Ai_{1-\gamma}(a^{-2/3}(1+ax))e^x $ gains us a factor $a$. This is   a consequence of the multiplication by $e^x$. We used the expansion of $\Hi_\gamma$ in this expansion. 
The derivatives  with respect to $p$ fall only on $Q$, and hence they are  easy to estimate.  
The Scorer functions are differentiable with respect to $\gamma$. This implies 
the statement on the differentiablity of the coefficients.  
\end{proof} 

As a consequence we have 
\[   |\langle u, Q_x \rangle (v^a_\gamma  - c_0 \Hi_{\gamma}(a^{-2/3}(1+ax)))| \le c a  e^{\kappa x} \]
hence 
\[ |\langle u, Q_x \rangle |v^a_\gamma | \le ca \left( \frac{\Hi_{\gamma}(a^{-\frac23} (1+ax)) }{ \Hi_\gamma(a^{-\frac23} ) }   +  \frac{e^{\kappa x}}{\sqrt{1-\kappa}} \right)  \]
for $ x\le  0 $ 
and similarly, for $ x>0$, 
 \begin{equation}    |\langle u, Q_x \rangle v^a | \le c a \left( a^{2}  \frac{\Gi_{\gamma}(a^{-2/3}(1+ax))}{\Gi_{-\gamma}(a^{-2/3})}  +  \frac{e^{-\kappa x}}{\sqrt{1-\kappa}} \right)  
\end{equation} 
In the sequel we will only rely on those two estimates, and not on the
full statement of Proposition \ref{firstest}.

\subsection{Bounds for the fixpoints and derivatives with respect to $x$}

After this warm-up we turn to the nonlinear term. Let 
\[ c_1 =   a^{-1/3} \int \Ai_{-\gamma}(a^{-2/3}(1+ay)) u^{p+1}(y) dy \]
and 
\[ d_1 = a^{-4/3} \int \Gi_{-\gamma}(a^{-2/3}(1+ay)) u^{p+1}(y) dy. \]
These integrals exist since $|u| \lesssim   w^a(x)$ . Using the bound 
\[ |u-Q| \le ca w^a(x) \]
of the previous section we see that 
\[   |u^{p+1} - Q^{p+1}| \le ca w^a_i(x) \]
and, as in the previous subsection  
\[  c_1=(1+O(a))  a^{-1/3} \Ai_{-\gamma}(a^{-2/3}) \int Q^{p+1}(y) dy.  \]
Thus 
\begin{equation}  c_1  \sim  ( \Hi_{\gamma}(a^{-2/3})^{-1}). \end{equation}
and similarly 
\begin{equation}  d_1 \sim  a^{-\frac{2(1+\gamma)}3}.  \end{equation}

The function $u^{p+1}$ decays sufficiently fast 
to repeat the argument of the last section. Thus 
\[ w_L(x): =  u(x) - v^a-  c_1  \Hi_{\gamma}(a^{-2/3}(1+ax))= \int_{-\infty}^x K_L^a(x,y) \partial_y u^{p+1}(y) dy \]
and  we also have the obvious integral representation for 
\[ w_R(x):= u(x) -v^a - d_1  \Gi_{\gamma}. \]

Thus we obtain the very rough estimate, using $p\ge 3$ and $\gamma$ close to $-1/2$  
\[ |w_L(x) | \le  c^{p+1}    \int_{-\infty}^x |\partial_y K_L(x,y)| (w^a(x))^{p+1} dx   \le c \left(\frac{\Hi_{\gamma}(a^{-2/3}(1+ax))}{\Hi_{\gamma}(a^{-2/3})}  \right)^4.     \]
 We put this information in the expansion. The oscillatory part of the kernel gives a small contribution when applied to $|x|^{-(\gamma+1)(p+1)}$
resp. $(\Hi_{\gamma}(a^{-2/2}(1+ax)))^{p+1}$ hence,
with $\omega= -(1+\gamma)(p+1)-p-2$

\begin{equation}  |w_L(x)| \lesssim 
\left\{ \begin{array}{ll}  
 a^{-1/3} |1+ a^{-2/3}(1+ax)|^{\omega} (\Hi_\gamma(a^{-2/3}))^{-p-1}& \text{ if } x\le -a^{-1}  \\[2mm] 
    |a^{2/3}+ (1+ax)|^{-1/2}       \left(   \frac{\Hi_\gamma (a^{-2/3}(1+ax))}{ \Hi_\gamma(a^{-2/3})}\right)^{p+1} 
& \text{ if } -a^{-1} \le x \le 0. \end{array} \right. . 
\end{equation}

Similarly we repeat the arguments from the last section on the right hand side. 
In a first step 
\[ |w_R(x)| \le c \left(e^{- x} + 
a \frac{\Gi_\gamma(a^{-2/3} (1+ax))}{\Gi_\gamma(a^{-2/3})}\right).
 \] 
We plug this into the integral operator. The exponential part with $\Ai$ and $\Hi$ 
reproduces the decay $ (Q+ c a(1+ax)^{-1-\gamma})^{p+1}$.  The second potentially large 
contribution is bounded by  
\[ \begin{split}  a \int_x^\infty (1+ay)^{-1+\gamma}  (Q+ ca(1+ax)^{-1-\gamma})^{-(p+1) } dy (1+ax)^{-1-\gamma} & \\  & \hspace{-5cm} \lesssim   a e^{-x}  + a^{p+1} (1+ax)^{-(p+1)\gamma-(p+2)} 
\end{split}
\]
 The exponential part again reproduces 
the decay and we arrive at 
\[ |w_R(x)| \le c    e^{-x} + a^{p+1} (1+ax)^{-(1+\gamma)(1+p)-1} \]
for $x >0$
All three  expansions remain correct under differentiation. We collect the estimates
in the following lemma. 
\begin{lemma} 
There exists $\varepsilon>0$ so that for 
\[ 0\le a \le \varepsilon, \quad \left|\frac12+\gamma\right| \le \varepsilon, \qquad 3\le p \le q \] 
 and the fixed point $u$ the following is true. 
Let 
\[ w_L=  u-v^a-c_1 \Hi_{\gamma}(a^{-2/3}(1+ax))-Q.  \]
Then 
\[  | \partial_x^k w_L  | 
\le c   
\left\{ \begin{array}{cl}
 \frac{\displaystyle a^{-1/3}|1+a^{-2/3}(1+ax)|^{-(p+1)\gamma-p-2}}{\displaystyle  (\Hi_\gamma(a^{-2/3}))^{p+1} }  & \text{ if } x \le -a^{-1} \\[3mm]
   (a^{2/3}+(1+ax))^{\frac{k-1}2} \left(\frac{\Hi_{\gamma}(a^{-2/3}(1+ax))}{ \Hi_\gamma(a^{-2/3})}  \right)^{p+1} & \text{ if } -a^{-1} \le x \le 0 
\end{array} \right. . 
\]
 and with 
\[ w_R =  u-v^a- ad_1 \Gi_{\gamma}(a^{-2/3}(1+ax))-Q \]
the estimate
\[   |\partial_x^k w_R | 
\le c_k a \Big[ e^{-x} +   a^{p+k} (1+ax))^{-(p+1)\gamma-p-2-k} \Big]  
\]
holds.
 The sum $(c_0+c_1) $ und $d_1$  are  bounded and bounded from below by a positive constant,
independent of $p$, $\gamma$ and $a$. 
Finally
\begin{equation} \label{posleft}      c^{-1}  \frac{\Hi_{\gamma}(a^{-2/3}(1+ax))}{\Hi_\gamma(a^{-\frac23) }}  \le u \le c   \frac{\Hi_{\gamma}(a^{-2/3}(1+ax))}{\Hi_\gamma(a^{-\frac23}) }
 \end{equation} 
if $x \le 0$ and 
\begin{equation} \label{posright}   c^{-1} (e^{- x} +  a (1+ax)^{-1-\gamma})  \le u \le c (e^{- x/2} + a (1+ax)^{-1-\gamma}) \end{equation} 
if $x \ge 0 $. 
\end {lemma} 

\begin{proof} Only the last two statements need to be shown. Since 
\[   |u-Q| \le a w^a \]
the statement is obvious for $ |x| \le   |\ln a|/2$. 

For $x \le -|\ln a|+R $ and $a$ sufficiently small the term 
$\frac{\Hi_\gamma(a^{-2/3}(1+ax))}{\Hi_\gamma(a^{-2/3})}$ becomes 
dominant and ensures positivity for those $x$. The same argument applies on the 
right hand side. 
\end{proof} 

 In particular $u$ is positive and bounded from below by the same type of bounds 
as from above.  

\subsection{Derivatives with respect to $\gamma$ and $a$}

The result of this subsection concludes the proof. 

\begin{prop} \label{sharpdecay} The fixed point $u$ is infinitely often  differentiable with respect to $x$, $a$, $\gamma$ and $p$ up to $a=0$. Moreover the estimate 
\[\begin{split}  |\partial_x^k \partial^l_a \partial_p^m \partial_\gamma^n u| & \\ & \hspace{-2cm} \lesssim  \left\{ 
\begin{array}{cl} 
  (a^ {-2/3}+|1+ax|)^{-1-\gamma-k-n} |\ln(2+|ax|)|^n  / \Hi_{\gamma}(a^{-2/3}) & \text{ if } x < -a^{-1} \\ 
\Hi_\gamma(a^{-2/3}(1+ax))/ \Hi_\gamma(a^{-2/3}) & \text{ if }  -a^{-1} \le x \le 0 \\
e^{-x} + a^{1+k }  (1+ax)^{-1- \gamma-k-n}|\ln (2+ax)|^n     & \text{ if } x \ge 0 \\
\end{array} \right. 
\end{split} 
\]
holds with a constant depending only on $k,l,m$ and $n$. 
\end{prop}

The bounds are exactly the bounds for the derivatives of 
\[ a (1+ax)^{-1-\gamma} \]
plus a Schwartz function, 
resp. for $x\le 0$
\[   \Hi_\gamma(a^{-2/3}(1+ax)) / \Hi_\gamma(a^{-2/3}) \]
Proposition \ref{sharpdecay} completes the proof of Theorem \ref{inverse}.

Despite considering  a nonsmooth nonlinearity the fixed point will be smooth. 
This is compatible with the nonregularity of the power function since the 
fixed point $u$   is positive.

\begin{proof} 
 The differentiation with respect to $p$ is  simpler than the differentiation with 
respect to $a$ and $ \gamma$, and we ignore it.  We differentiate  
\[  a( (1+\gamma)u + xu_x) -u_{xxx}+u_x + \partial_x (|u|^p u + \langle u, Q_x\rangle Q_x) = 0 \]
with respect to $\gamma$ formally and  denote the derivative again by
$\dot u$. It satisfies 
\[  a( (1+\gamma)\dot u + x\dot u_x) -\dot u_{xxx}+\dot u_x + \partial_x ((p+1)|u|^p \dot u + \langle \dot u, Q_x\rangle Q_x) =   -au   \]
By Proposition \ref{invertibility} the linear operator is invertible, and we want  estimate $\dot u$ in terms 
of $u$. However,  we do not have the bound $|u| \le w^a_i$ for $ |x| \le -a^{-1}$ since there $w_i^a $ is not bounded by $w^a$.  

We choose a smooth monotone function $\eta_+$, supported in $[-1,\infty)$ and identically one in $[1,\infty)$. Let $\eta_(x) = 1-\eta_+(x)$. 
We denote 
\[ \dot \Hi_\gamma = \frac{\d}{\d\gamma} \Hi_\gamma \]
and 
\[ \dot \Gi_\gamma = \frac{\d}{\d\gamma} \Gi_\gamma \]
Then 
\[  a( (1+\gamma) \dot \Hi_\gamma + x \dot \Hi_\gamma') - \dot \Hi_{\gamma}''' + \dot \Hi_\gamma' = -a  \Hi_\gamma \]
Let 
\[ \dot v = \dot u - (c_0+c_1) \eta_-\dot \Hi_\gamma- a (d_0+ d_1)\eta_+ \dot \Gi_\gamma \]
Then 
\[ \begin{split} 
 a(1+\gamma) \dot v + x\dot v_x) -\dot v_{xxx}+ \dot v_x + \partial_x ((p+1)|u|^p \dot v + \langle \dot v, Q_x\rangle Q_x)  \hspace{-7cm} & \\   = & \phi
 -a(u-(c_0+c_1 ) \eta_-\Hi_\gamma - a(d_0+d_1) \eta_+\Gi_\gamma )  
\\ &  - \partial_x((p+1) |u|^p \Big[(c_0+c_1) \eta_-\dot \Hi_\gamma + a(d_0+d_1)\eta_+ \dot \Gi_\gamma \Big]  
\\ & + \langle  (c_0+c_1) \eta_-\dot \Hi_\gamma + a(d_0+d_1) \eta_+\dot \Gi_\gamma , Q_x \rangle Q_x  .               
\end{split} 
 \]
for some smooth function $\phi$ supported in $[-1,1]$. 
The right hand side decays sufficiently fast to apply Proposition \ref{invertibility}. To justify this formal argument we use finite differences. 
Continuity with respect to all parameters is obvious. This argument can be iterated. 

Similarly we deal with derivatives with respect to $a$. 
The partial derivatives  $\partial_a^n \Gi_\gamma (a^{-1/2}(1+ax))$   
behave similarly as  
\[ \partial_a^n (1+ax)^{-1-\gamma}= c_{\gamma,n} \frac{ x^n}{(1+ax)^{-n-\gamma} } . \]
Again Proposition \ref{invertibility} implies differentiability with respect to $a$, for $a>0$, but this time we have to use weights with 
$k>0$.

There is basically no difference in applying this argument 
to the derivative with respect to $a$, $p$ or $x$, using crucially the estimate \eqref{posleft} and \eqref{posright}.  

 \end{proof}

\subsection{Expansion of the selfsimilar solution}

The argument above gives information on  the asymptotics of the self similar solutions which we state below. 

\begin{prop}
Let $u= u(a)$ be the selfsimilar solution orthogonal to $Q_x$. Then exists a unique  expansion
\[  \left| u(x) - a(1+ax)^{-2/p} \sum_{j=0}^N d_j(a)  (1+ax)^{-3j}\right| \le a c_N (1+ax)^{-3N-3}        \]
for $x >1$, where $d_j$ are bounded uniformly in $a$, and $c_N$ is independent of $a$ and
\[  \left| u(x) -  (\Hi_\gamma(a^{-2/3}))^{-1}|1+ax|^{-2/p} \sum_{j=0}^N d_j(a)  (1+ax)^{-3j}\right| \le a c_N (1+|ax|)^{-3N-3}        \]
for $ x < -2a^{-1} $. 
\end{prop}

\section{A numerical simulation} 

The selfsimilar solutions have been computed numerically by N. Strunk in his diploma thesis.  
The first curve shows $1/p$ as a function of $a$. There is a small artefact near $a=0.1$. 

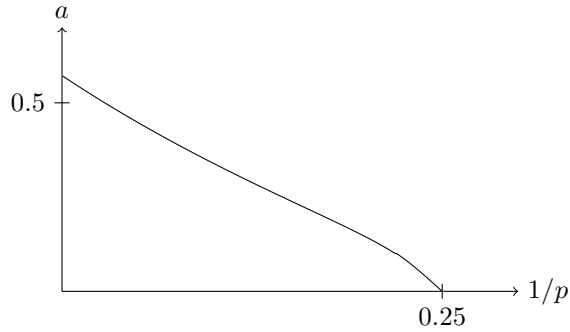
\begin{center} 
 
\end{center} 

The next plots show the selfsimilar solution $u$ as a function of $x$ for various values of $a$.

%


\printbibliography

\end{document}